\documentclass[final,12p,times,reqno]{amsart}

\usepackage[colorlinks=true,linkcolor=black, citecolor=blue, urlcolor=blue]{hyperref}
\usepackage{amsfonts,latexsym}
\usepackage{amssymb}
\usepackage{amsmath, times}
\author[N. Attia,\; B. Selmi]
{Najmeddine Attia,\; Bilel Selmi}  % in alphabetical order

 \bigskip  \bigskip
\address{\newline Analysis, Probability and Fractals Laboratory LR18ES17\\ Department of Mathematics, Faculty of
Sciences of Monastir, University of Monastir, 5000-Monastir,
Tunisia}

\email{najmeddine.attia@gmail.com}

\email{ bilel.selmi@fsm.rnu.tn}


\newcommand{\R}{\mathbb R}

\newcommand{\N}{\mathbb N}

\newtheorem{theorem}{Theorem}
\newtheorem{lemma}{Lemma}
\newtheorem{proposition}{Proposition}
\newtheorem{corollary}{Corollary}
\newtheorem{definition}{Definition}
\newtheorem{remark}{Remark}

\DeclareMathOperator{\supp}{supp} \DeclareMathOperator{\Dim}{Dim}
 \numberwithin{equation}{section}

\title[ A multifractal formalism for Hewitt-Stromberg measures]{A multifractal formalism for Hewitt-Stromberg measures}

\begin{document}

\maketitle
\begin{abstract}
 %One important thing which should be noted is that there are many measures for which the multifractal formalism does not hold.
In the present work, we give a new {\it multifractal formalism} for
which the classical multifractal formalism does not hold. We
precisely  introduce and study a multifractal formalism based on the
Hewitt-Stromberg measures and that this formalism is completely
parallel to Olsen's multifractal formalism which based on the
Hausdorff and packing measures.

\end{abstract}

\maketitle

%%%%%%%%%%%%%%%%%%%%%%%%%%%%%%%%%%%%%%%%%%%%%%%%%%%%%%%%%%%%%%%%%%%%%%%%%%%%%%%%%%%%%%%%%%%%%%%%%%%%%%%%%%%%%%%%%%%%%%%%%%%%ù
                                                                  \section{Introduction}
%%%%%%%%%%%%%%%%%%%%%%%%%%%%%%%%%%%%%%%%%%%%%%%%%%%%%%%%%%%%%%%%%%%%%%%%%%%%%%%%%%%%%%%%%%%%%%%%%%%%%%%%%%%%%%%%%%%%%%%%%%%%%%%%%%%%%

In certain circumstances, a measure $\mu$ gives rise to sets of
points where $\mu$ has local density of exponent $\alpha$. The
dimensions of these sets indicate the distribution of the
singularities of the measure. To be more precise, for a finite
measure $\mu$ on $\mathbb{R}^n$, the pointwise dimension at $x$ is
defined as follows
$$
{\alpha}_{\mu}(x)=\lim_{r\to 0}\frac{\log \mu(B(x,r))}{\log r},
$$
whenever this limit exists. For $\alpha\geq0$, define
$$
E(\alpha)=\Big\{x\in\supp\mu\;\big|\;\quad
{\alpha}_{\mu}(x)=\alpha\Big\}
$$
where $B(x,r)$ is the closed ball with center $x$ and radius $r$.
The set $E(\alpha)$ may be thought of as the set where the {\it
local dimension} of $\mu$ equals $\alpha$ or as a multifractal
component of $\supp\mu$. The main problem in multifractal analysis
is to estimate the size of $E(\alpha)$. This is done by calculating
the functions $$f_\mu(\alpha)=\dim_H(E(\alpha))\quad\text{ and}\quad
F_\mu(\alpha)=\dim_P(E(\alpha))\;\;\text{ for} \;\;\alpha\geq0.$$
These functions are generally known as the {\it multifractal
spectrum} of $\mu$ or {\it the singularity spectrum } of the measure
$\mu$. One of the main problems in multifractal analysis is to
understand the multifractal spectrum and the R\'{e}nyi dimensions
and their relationship with each other. During the past 25 years
there has been an enormous interest in computing the multifractal
spectra of measures in the mathematical literature. Particularly,
the multifractal spectra of various classes of measures in Euclidean
space $\mathbb{R}^n$ exhibiting some degree of self-similarity have
been computed rigorously. The reader can be referred to the paper
\cite{Ol1}, the textbooks \cite{F, Pe} and the references therein.
Some heuristic arguments using techniques of Statistical Mechanics
(see \cite{HJKPS}) show that the singularity spectrum should be
finite on a compact interval, noted by $\text{Dom}(\mu)$, and is
expected to be the Legendre transform conjugate of the
$L^q$-spectrum, given by
$$
\tau_\mu(q)=\displaystyle\lim_{r\to0} \frac{\log \left(\sup\left\{
\displaystyle \sum_i \mu(B(x_i,r))^q\right\}\right)}{-\log r}
$$
where the supremum is taken over all centered packing
$\big(B(x_i,r)\big)_i$ of $\supp\mu$. That is, for all $\alpha\in
\text{Dom}(\mu)$,
\begin{eqnarray}\label{formalismtau}
f_\mu(\alpha)=\inf_{q\in\mathbb{R}} \Big\{\alpha
q+\tau_\mu(q)\Big\}=\tau_\mu^*(\alpha).
\end{eqnarray}

The multifractal formalism \eqref{formalismtau} has been proved
rigorously for random and non-random self-similar measures \cite{AP,
Da2, Ol1, O3, OlS}, for self-conformal measures \cite{Feng, Feng1,
Feng2, Feng3, KG, Pa}, for self-affine measures \cite{Barral,
Barral-Feng, Barral-Mensi, Barral-Seuret, Falconer, Falconer1, Ki,
O4} and for Moran measures \cite{W, W1, W3, W4}. We note that the
proofs of the multifractal formalism \eqref{formalismtau} in the
above-mentioned references \cite{AP, BBJ, BB, BBH, BMP,  Ki, KG,
Ol1, O3, O4, Pa} are all based on the same key idea. The upper bound
for $f_\mu(\alpha)$ is obtained by a standard covering argument
(involving Besicovitch's Covering Theorem or Vitali's Covering
Theorem). However, its lower bound is usually much harder to prove
and is related to the existence of an auxiliary measure (Gibbs
measures) which is supported by the set to be analysed. In an
attempt to develop a general theoretical framework for studying the
multifractal structure of arbitrary measures, Olsen \cite{Ol1},
Pesin \cite{Pes} and Peyri\`{e}re \cite{JP} suggested various ways
of defining an auxiliary measure in a very general setting. This
formalism was motivated by Olsen's wish to provide a general
mathematical setting for the ideas presented by the physicists
Halsey et al. in their seminal paper \cite{HJKPS}. In fact, they
have been interested in the concept of multifractal spectrum, that
is an interesting geometric characteristic for discrete and
continuous models of statistical physics. An important thing which
should be noted is that there are many measures for which the
multifractal formalism does not hold (some examples could be found
in \cite{BJ, BBH, Ol1, W4}). An imported question, in which several
theorists are interested, is: can we find a necessary and sufficient
condition for the multifractal formalism to hold? Another one, asked
by Olsen in \cite{Ol1} is: which functions give more information
about a multifractal measure, the dimension functions $b_\mu$ and
$B_\mu$ or the spectra functions $f_\mu$ and $F_\mu$? Olsen gives
examples of measures where the dimension functions can be used to
split measures which have the same spectrum. In doing this, he
implicitly suggests that a return to the physicists' original idea
of calculating the moments of multifractal measures may be the best
way to characterize them.  It always needs some extra conditions to
obtain a minoration for the dimensions of the level sets
$E(\alpha)$. Olsen proved the following statement.
\begin{theorem}\cite{Ol1}
Let $\mu$ be a Borel probability measure on $\mathbb{R}^n$. Define
$\underline{\alpha}=\displaystyle\sup_{0<q} -\frac{{{
b}_\mu(q)}}{q}$ and $\overline{\alpha}=\displaystyle\inf_{0>q}
-\frac{{{ b}_\mu(q)}}{q}$. Then,
$$
\dim_H (E(\alpha))\leq { b}_\mu^*(\alpha)\quad\text{and}\quad \dim_P
(E(\alpha))\leq { B}_\mu^*(\alpha)\quad\text{for all}\quad \alpha\in
(\underline{\alpha}, \;\overline{\alpha}).
$$
\end{theorem}
In general, such a minoration is related to the existence of an
auxiliary measure which is supported by the set to be analyzed.
Olsen also gives a result in such a way and supposes the existence
of a Gibbs' measure (see \cite{Ol1}) at a state $q$ for the measure
$\mu$, i.e., the existence of a measure $\nu_q$ on $\supp\mu$ and
constants $C> 0$, $\delta > 0$ such that for every $x \in \supp\mu$
and every $0 < r < \delta$,
$$
\frac1C ~\mu(B(x,r))^q ~(2r)^{{ B}_\mu(q)}\leq \nu_q(B(x,r))\leq C~
\mu(B(x,r))^q~ (2r)^{{ B}_\mu(q)}
$$
to conclude that
$$
\dim_H( E(\alpha))=\dim_P (E(\alpha))={ b}_\mu^*(\alpha)={
B}_\mu^*(\alpha),\quad\text{where}\quad \alpha=-{ B}_\mu'(q).
$$
In general, one needs some degree of similarity to prove the
existence of Gibbs measures. For example, in dynamic contexts, the
existence of such measures are often natural. For this reason, Ben
Nasr et al.  in \cite{BBJ, BJ, BB, BBH} improved Olsen's result and
proposed a new sufficient condition that gives the lower bound. For
more details and backgrounds on multifractal analysis as well as
their applications the readers may be referred also to the following
essential references \cite{NB1, NBC, Co, DB, BD1, BD2, MMB, TO1, O2,
Ol, SH1, SB, SBbb, SB1, BBSS, W, W1, W3, W4}.

 \bigskip \bigskip
In \cite{BJ, BBH, Ol1, SH1, W4}, the authors provided some examples
for which the classical multifractal formalism does not hold.
Indeed, for such examples, the functions $b_\mu$ and $B_\mu$ differ
and $\dim_H (E(\alpha))$ and $\dim_P(E(\alpha))$ are given
respectively  by the Legendre transform of $b_\mu$ and $B_\mu$.
Motivated by the above papers, the authors in \cite{NB} introduced
new metric outer measures (multifractal analogues of the
Hewitt-Stromberg measure) $\mathsf{ H}^{q,t}_{\mu}$ and  $\mathsf{ P
}^{q,t}_{\mu}$ lying between the multifractal Hausdorff measure
${\mathcal H}^{q,t}_{\mu}$  and the multifractal  packing measure
${\mathcal P}^{q,t}_{\mu}$, and they used the multifractal density
theorems to prove the decomposition theorem for the regularities of
these measures. In the present paper, we give a new {\it
multifractal formalism} for which the functions $b_\mu$ and $B_\mu$
differ. Actually, the main aim of this work is to introduce and
study a multifractal formalism based on the Hewitt-Stromberg
measures. However,  we point out that this formalism is completely
parallel to Olsen's multifractal formalism introduced in \cite{Ol1}
which based on the Hausdorff and packing measures. Then, we prove
that the lower and upper multifractal Hewitt-Stromberg functions
$\mathsf{b}_\mu$ and $\mathsf{B}_\mu$ are intimately related to the
spectra functions. More precisely, we have
$$
\mathsf{f}_\mu(\alpha):=\underline{\dim}_{MB}(E(\alpha)) \le
{\mathsf b}_\mu^*(\alpha)\quad\text{and}\quad
\mathsf{F}_\mu(\alpha):=\overline{\dim}_{MB}(E(\alpha)) \le {\mathsf
B}_\mu^*(\alpha)\quad\text{for some}\quad \alpha\geq0.
$$
Here $\underline{\dim}_{MB}$ and $\overline{\dim}_{MB}$ denote,
respectively, the lower and the upper Hewitt-Stromberg dimension
(the lower and the upper  modified box-counting dimension), see
Section \ref{sec2.2} for precise definitions of this. One of our
purposes of this paper is to show the following result: if
$\mathsf{H}_\mu^{q, {\mathsf b}_\mu(q)} ( E(- {\mathsf b_\mu}' (q))
) > 0$, then
 $$
 \underline{\dim}_{MB}\left(E\big( - {\mathsf b}_\mu^{'}(q)\big)\right)=  {\mathsf b}_\mu^*\big( - {\mathsf b}_\mu^{'}(q)\big)
 $$
and, if $\mathsf{P}_\mu^{q, {\mathsf B}_\mu(q)} ( E(- {\mathsf
B_\mu}' (q)) ) > 0$, then
 $$
   \overline{\dim}_{MB}\left(E\big(  - {\mathsf B}_\mu^{'}(q)\big)\right) =  {\mathsf B}_\mu^*\big( - {\mathsf B}_\mu^{'}(q)\big).
 $$
Moreover, we  describe a sufficient condition leading to the
equalities
$$
\mathsf{f}_\mu(\alpha)=
\mathsf{F}_\mu(\alpha)={F}_\mu(\alpha)\quad\text{for some}\quad
\alpha\geq0.
$$
Specifically, if we assume that $\mathsf{ H}^{q, \mathsf
B_{\mu}(q)}_{\mu}(\supp\mu)>0$, then
$$
\underline{\dim}_{MB}\left(E\big(- {\mathsf
B}_\mu^{'}(q)\big)\right)=\overline{\dim}_{MB}\left(E\big(- {\mathsf
B}_\mu^{'}(q)\big)\right) = \mathsf  b^*_{\mu}(- {\mathsf
B}_\mu^{'}(q))=\mathsf B^*_{\mu}(- {\mathsf B}_\mu^{'}(q)).
$$
We also observe that this sufficient condition is very close to
being a necessary and sufficient one, see Theorem \ref{SNC}. In
particular, we deal with the case where the lower and upper
multifractal Hewitt-Stromberg functions $\mathsf{b}_\mu$ and
$\mathsf{B}_\mu$ do not necessarily coincide, see Theorem
\ref{newform}.

\bigskip
We will now give a brief description of the organization of the
paper. In the next section we recall the definitions  of the various
fractal and multifractal dimensions and measures investigated in the
paper. The definitions of the Hausdorff and packing measures and the
Hausdorff and packing dimensions are recalled in Section
\ref{sec2.1}, and the definitions of the Hewitt-Stromberg measures
are recalled in Section \ref{sec2.2}, while the definitions of the
Hausdorff and packing measures are well-known, we have,
nevertheless, decided to include these-there are two main reasons
for this: firstly, to make it easier for the reader to compare and
contrast the Hausdorff and packing measures with the less well-known
Hewitt-Stromberg measures, and secondly, to provide a motivation for
the Hewitt-Stromberg measures. Section \ref{sec2.3} recalls the
multifractal formalism introduced in \cite{Ol1}. In Section
\ref{sec2.4} we recall the definitions of the multifractal
Hewitt-Stromberg measures and separator functions, and study their
properties. Section \ref{sec2.5} recalls earlier results on the
values of the multifractal Hausdorff measure, the multifractal
packing measure, the multifractal Hewitt-Stromberg measures and
separator functions; this discussion is included in order to
motivate our main results presented in Section \ref{sec3}. Section
\ref{sec4} contains concrete examples related to these concepts. The
paper is concluded with Section \ref{sec5} that, lists some open
problems.

%%%%%%%%%%%%%%%%%%%%%%%%%%%%%%%%%%%%%%%%%%%%%%%%%%%%%%%%%%%%%%%%%%%%%%%%%%%%%%%%%%%%%%%%%%%%%%%%%%%%%%%%%%%%%%%%%%%%%%%%%%%%%%%%%%%%
                                                               \section{Preliminaries and  statements of results}
%%%%%%%%%%%%%%%%%%%%%%%%%%%%%%%%%%%%%%%%%%%%%%%%%%%%%%%%%%%%%%%%%%%%%%%%%%%%%%%%%%%%%%%%%%%%%%%%%%%%%%%%%%%%%%%%%%%%%%%%%%%%%%%%%%%%

\subsection{Hausdorff measure, packing measure and dimensions}\label{sec2.1} While
the definitions of the Hausdorff and packing measures and the
Hausdorff and packing dimensions are well-known, we have,
nevertheless, decided to briefly recall the definitions below. There
are several reasons for this: firstly, since we are working in
general metric spaces, the different definitions that appear in the
literature may not all agree and for this reason it is useful to
state precisely the definitions that we are using; secondly, and
perhaps more importantly, the less well-known Hewitt-Stromberg
measures (see Section \ref{sec2.2}) play an important part in this
paper and to make it easier for the reader to compare and contrast
the definitions of the Hewitt-Stromberg measures and the definitions
of the Hausdorff and packing measures it is useful to recall the
definitions of the latter measures; and thirdly, in order to provide
a motivation for the Hewitt-Stromberg measures.

\bigskip
Let $X$ be a metric space, $E \subseteq X$ and $t>0$. The Hausdorff
measure is defined, for $\varepsilon>0$, as follows
$$
\mathcal{H}_\varepsilon^t(E)= \inf\left\{\sum_i
\Big(\text{diam}(E_i)\Big)^t\Big|\;\; E\subseteq\bigcup_i
E_i,\;\;\text{diam}(E_i)<\varepsilon \right\}.
$$
This allows to define first the $t$-dimensional Hausdorff measure
$\mathcal{H}^t(E)$ of $E$ by
$$
\mathcal{H}^t(E)=\sup_{\varepsilon>0}\mathcal{H}_\varepsilon^t(E).
$$
Finally, the Hausdorff dimension $\dim_H(E)$ is defined by
$$
\dim_H(E)=\sup\Big\{t\geq0\;\big|\; \mathcal{H}^t(E)=+\infty\Big\}.
$$

The packing measure  is defined, for $\varepsilon> 0$, as follows
\begin{eqnarray*}
\overline{\mathcal{P}}_\varepsilon^t(E)= \sup\left\{\sum_i
\Big(2r_i\Big)^t\right\},
\end{eqnarray*}
where the supremum is taken over all closed balls $\Big( B(x_i, r_i)
\Big)_i \; \text{such that}\; r_i \leq \varepsilon\; \text{and
with}\;  x_i \in E \; \text{and}\; d(x_i, x_j)\geq \frac{r_i+r_j}2\;
\text{for}\; i\neq j$. The $t$-dimensional packing pre-measure
$\overline{\mathcal{P}}^t(E)$ of $E$ is now defined by
$$
\overline{\mathcal{P}}^t(E)=\sup_{\varepsilon>0}\overline{\mathcal{P}}_\varepsilon^t(E).
$$
This makes us able to define the $t$-dimensional packing measure
${\mathcal{P}}^t(E)$ of $E$ as
$$
{\mathcal{P}}^t(E)=\inf\left\{\sum_i\overline{\mathcal{P}}^t(E_i)\;\Big|\;\;E\subseteq\bigcup_i
E_i\right\},
$$
and the packing dimension $\dim_P(E)$ is defined by
$$
\dim_P(E)=\sup\Big\{t\geq0\;\big|\; \mathcal{P}^t(E)=+\infty\Big\}.
$$

\subsection{Hewitt-Stromberg measures and dimensions}\label{sec2.2}
Hewitt-Stromberg measures were introduced  in \cite[Exercise
(10.51)]{HeSt}. Since then, they  have  been investigated  by
several authors, highlighting their importance in the study of local
properties of fractals and products of fractals. One can cite, for
example \cite{Ha1, Ha2, JuMaMiOlSt, Olll2, Zi}. In particular,
Edgar's textbook \cite[pp. 32-36]{Ed} provides an excellent and
systematic introduction to these  measures. Such measures appear
also appears explicitly, for example, in  Pesin's monograph
\cite[5.3]{Pe} and implicitly in Mattila's text \cite{Mat}. One of
the purposes of this paper is to define and study a class of natural
multifractal generalizations of the Hewitt-Stromberg measures. While
Hausdorff and packing measures are defined using coverings and
packings by families of sets with diameters less than a given
positive number $\varepsilon$, say, the Hewitt-Stromberg measures
are defined using packings of balls with a fixed diameter
$\varepsilon$. For $t>0$, the Hewitt-Stromberg pre-measures are
defined as follows,
$$
\overline{\mathsf{H}}^t(E)=\liminf_{r\to0} N_r(E) \;(2r)^t
$$
and
$$
\overline{\mathsf{P}}^t(E)=\limsup_{r\to0} M_r(E) \;(2r)^t,
$$
where the  covering number $N_r(E)$ of $E$ and the packing number
$M_r(E)$ of $E$ are given by
$$
N_r(E)=\inf\left\{\sharp\{I\}\;\Big|\; \Big( B(x_i, r) \Big)_{i\in
I} \; \text{is a family of closed balls with}\; x_i \in E \;
\text{and}\; E\subseteq \bigcup _i B(x_i, r)\right\}
$$
and
$$
M_r(E)=\sup\left\{\sharp\{I\}\; \Big|\; \Big( B(x_i, r_i)
\Big)_{i\in I} \; \text{is a family of closed balls with}\; x_i \in
E \; \text{and}\; d(x_i, x_j)\geq r\; \text{for}\; i\neq j\right\}.
$$

Now, we define the lower and upper $t$-dimensional Hewitt-Stromberg
measures,  which we denote respectively  by $\mathcal{\mathsf
H}^t(E)$ and $\mathcal{\mathsf P}^t(E)$,
 as follows
$$
{\mathsf{H}}^t(E)=\inf\left\{\sum_i\overline{\mathsf{H}}^t(E_i)\;\Big|\;\;E\subseteq\bigcup_i
E_i\right\}
$$
and
$$
{\mathsf{P}}^t(E)=\inf\left\{\sum_i\overline{\mathsf{P}}^t(E_i)\;\Big|\;\;E\subseteq\bigcup_i
E_i\right\}.
$$

We recall some basic inequalities satisfied by the Hewitt-Stromberg,
the Hausdorff  and the packing measure (see \cite[Proposition
2.1]{JuMaMiOlSt, Olll2})
$$
\overline{\mathsf{H}}^t(E)\leq \overline{\mathsf{P}}^t(E)\leq
\overline{\mathcal{P}}^t(E)
$$
and
$$
{\mathcal{H}}^t(E)\leq{\mathsf{H}}^t(E)\leq {\mathsf{P}}^t(E)\leq
{\mathcal{P}}^t(E).
$$
The lower and upper Hewitt-Stromberg dimension
$\underline{\dim}_{MB}(E)$ and  $\overline{\dim}_{MB}(E)$ are
defined by
$$
\underline{\dim}_{MB}(E)=\inf\Big\{t\geq0\;\Big|\;
{\mathsf{H}}^t(E)=0\Big\}=\sup\Big\{t\geq0\;\Big|\;
{\mathsf{H}}^t(E)=+\infty\Big\}
$$
and
$$
\overline{\dim}_{MB}(E)=\inf\Big\{t\geq0\;\big|\;
{\mathsf{P}}^t(E)=0\Big\}=\sup\Big\{t\geq0\;\big|\;
{\mathsf{P}}^t(E)=+\infty\Big\}.
$$
The lower and upper box dimensions, denoted by
$\underline{\dim}_{B}(E)$ and $\overline{\dim}_{B}(E)$,
respectively, are now defined by
$$
\underline{\dim}_{B}(E)=\liminf_{r\to0}\frac{\log N_r(E)}{-\log
r}=\liminf_{r\to0}\frac{\log M_r(E)}{-\log r}
$$
and
$$
\overline{\dim}_{B}(E)=\limsup_{r\to0}\frac{\log N_r(E)}{-\log
r}=\limsup_{r\to0}\frac{\log M_r(E)}{-\log r}.
$$

These dimensions satisfy the following inequalities,
$$
\dim_H(E)\leq\underline{\dim}_{MB}(E)\leq\overline{\dim}_{MB}(E)\leq\dim_P(E),
$$
$$
\dim_H(E)\leq \dim_P(E)\leq\overline{\dim}_{B}(E)
$$
and
$$\dim_H(E)\leq \underline{\dim}_{B}(E)\leq\overline{\dim}_{B}(E).
$$
The reader is referred to \cite{FKJ} for an excellent discussion of
the Hausdorff dimension, the packing dimension, lower and upper
Hewitt-Stromberg dimension and the box dimensions. In particular, we
have (see \cite{FKJ, MiOl})
$$
\underline{\dim}_{MB}(E)= \inf \left\{ \sup_{i}
\underline{\dim}_{B}(E_i)\;\Big|\;E\subseteq \bigcup_i E_i, \;\; E_i
\;\; \text{are bounded in } \; X\right\}
$$
and
$$
\overline{\dim}_{MB}(E)= \inf \left\{ \sup_{i}
\overline{\dim}_{B}(E_i) \;\Big|\; E\subseteq \bigcup_i E_i, \;\;
E_i \;\; \text{are bounded in } \; X\right\}.
$$

%%%%%%%%%%%%%%%%%%%%%%%%%%%%%%%%%%%%%%%%%%%%%%%%%%%%%%%%%%%%%%%%%%%%%%%%%%%%%%%%%%
\subsection{Multifractal Hausdorff measure and packing
measure}\label{sec2.3}
%%%%%%%%%%%%%%%%%%%%%%%%%%%%%%%%%%%%%%%%%%%%%%%%%%%%%%%%%%%%%%%%%%%%%%%%%%%%%%%%%%

We start by introducing the generalized centered Hausdorff measure
${\mathcal H}^{q,t}_{\mu}$ and the generalized packing measure
${\mathcal P}^{q,t}_{\mu}$. We fix an integer $n\geq 1$ and denote
by $\mathcal{P}(\mathbb{R}^n)$ the family of compactly supported
Borel probability measures on $\mathbb{R}^n$. Let
$\mu\in\mathcal{P}(\mathbb{R}^n)$,  $q, t\in\mathbb{R}$, $E
\subseteq{\mathbb R}^n$ and $\delta>0$. We define the generalized
packing pre-measure by
 $$
\overline{{\mathcal P}}^{q,t}_{\mu}(E)
=\displaystyle\inf_{\delta>0}\displaystyle \sup\left\{\sum_i
\mu\big(B(x_i,{r_i})\big)^q
(2r_i)^t\;\Big|\;\Big(B(x_i,r_i)\Big)_i\;\text{is a centered}\;
\delta\text{-packing of}\; E\right\}.
 $$
In a similar way, we define the generalized Hausdorff pre-measure by
$$\overline{{\mathcal H}}^{q,t}_{\mu}(E) =\displaystyle\sup_{\delta>0} \displaystyle\inf\left\{ \sum_i
\mu\big(B(x_i,r_i)\big)^q(2r_i)^t\;\Big|\;\Big(B(x_i,r_i)\Big)_i\;\text{is
a centered}\; \delta\text{-covering of}\; E\right\},
 $$
with the conventions $0^q = \infty$ for $q\leq0$ and $0^q = 0$ for
$q>0$.

 \bigskip
The function $ \overline{{\mathcal H}}^{q,t}_{\mu}$ is
$\sigma$-subadditive but not increasing and the function
$\overline{{\mathcal P}}^{q,t}_{\mu}$ is increasing but not
$\sigma$-subadditive. That is the reason for which  Olsen introduced
the following modifications of the generalized Hausdorff and packing
measures ${\mathcal H}^{q,t}_{\mu}$  and ${\mathcal P}^{q,t}_{\mu}$:
$$
{\mathcal H}^{q,t}_{\mu}(E)=\displaystyle\sup_{F\subseteq
E}\overline{{\mathcal H}}^{q,t}_{\mu} (F)\quad\text{and}\quad
{\mathcal P}^{q,t}_{\mu}(E) = \inf_{E \subseteq \bigcup_{i}E_i}
\sum_i \overline{\mathcal P}^{q,t}_{\mu}(E_i).
 $$

 \bigskip \bigskip
The functions ${\mathcal H}^{q,t}_{\mu}$ and ${\mathcal
P}^{q,t}_{\mu}$ are metric outer measures and thus measures on the
Borel family of subsets of $\mathbb{R}^n$. Moreover, there exists an
integer $\xi\in\mathbb{N}$, such that ${\mathcal
H}^{q,t}_{\mu}\leq\xi{\mathcal P}^{q,t}_{\mu}.$ The measure
${\mathcal H}^{q,t}_{\mu}$ is of course a multifractal
generalization of the centered Hausdorff measure, whereas ${\mathcal
P}^{q,t}_{\mu}$ is a multifractal generalization of the packing
measure. In fact, it is easily seen that, for  $t\geq0$,  one has
$$2^{-t} {\mathcal H}^{0,t}_{\mu}\leq {\mathcal H}^{t}\leq {\mathcal
H}^{0,t}_{\mu}\quad\text{and}\quad{\mathcal P}^{0,t}_{\mu}={\mathcal
P}^{t},$$ where ${\mathcal H}^{t}$ and ${\mathcal P}^{t}$ denote
respectively the $t$-dimensional Hausdorff and  $t$-dimensional
packing measures.

 \bigskip
We now define the family of doubling measures. For
$\mu\in\mathcal{P}(\mathbb{R}^n)$ and $a>1$, we write
 $$
P_a(\mu)=\limsup_{r\searrow0}\left( \sup_{x\in\supp\mu}
\displaystyle\frac{\mu\big(B(x,ar)\big)}{\mu\big(B(x,r)\big)}\right).
 $$
We say that the measure $\mu$ satisfies the doubling condition if
there exists $a>1$ such that $P_a(\mu)<\infty$. It is easily seen
that the exact value of the parameter $a$ is unimportant:
$$P_a(\mu)<\infty, \;\;\text{for some}\;\;  a>1\;\;\text{ if and only
if}\;\; P_a(\mu)<\infty,\;\;\text{ for all}\;\; a>1.$$ Also, we
denote by $\mathcal{P}_D( \mathbb{R}^n) $ the family of Borel
probability measures on $\mathbb{R}^n$ which satisfy the doubling
condition. We can cite as classical examples of doubling measures,
the self-similar measures and the self-conformal ones \cite{Ol1}. In
particular, if $\mu \in \mathcal{P}_D(\mathbb{R}^n)$  then
${\mathcal H}^{q,t}_{\mu}\leq {\mathcal P}^{q,t}_{\mu}.$

 \bigskip
The measures ${\mathcal H}^{q,t}_{\mu}$ and ${\mathcal
P}^{q,t}_{\mu}$ and the pre-measure ${\overline{\mathcal
P}}^{q,t}_{\mu}$ assign in a usual way a multifractal dimension to
each subset $E$ of $\mathbb{R}^n$. They are respectively denoted by
$\dim_{\mu}^q(E)$, $\Dim_{\mu}^q(E)$ and $\Delta_{\mu}^q(E)$ (see
\cite{Ol1}) and satisfy
 $$
\begin{array}{lllcr}
\dim_{\mu}^q(E) &=&\inf \Big\{ t\in\R\;\big|\; \quad {\mathcal
H}^{{q},t}_{\mu}(E) =0\Big\}=\sup \Big\{ t\in\R\;\big|\; \quad
{\mathcal
H}^{{q},t}_{\mu}(E) =+\infty\Big\}, \\ \\
\Dim_{\mu}^q(E) &=& \inf \Big\{  t\in\R\;\big|\; \quad {\mathcal
P}^{{q},t}_{\mu}(E) =0\Big\}=\sup \Big\{  t\in\R\;\big|\; \quad
{\mathcal
P}^{{q},t}_{\mu}(E) =+\infty\Big\}, \\ \\
\Delta_{\mu}^q(E) &=& \inf \Big\{ t\in\R\;\big|\; \quad
\overline{\mathcal P}^{{q},t}_{\mu}(E) =0\Big\}=\sup\Big\{
t\in\R\;\big|\; \quad \overline{\mathcal P}^{{q},t}_{\mu}(E)
=+\infty\Big\}.
\end{array}
 $$

The number $\dim_{\mu}^q(E)$ is an obvious multifractal analogue of
the Hausdorff dimension $\dim_H(E)$ of $E$ whereas $\Dim_{\mu}^q(E)$
and $\Delta_{\mu}^q(E)$ are obvious multifractal analogues of the
packing dimension $\dim_P(E)$ and the pre-packing dimension
$\Delta(E)$ of $E$ respectively. In fact, it follows immediately
from the definitions that
$$
\dim_H(E)=\dim_{\mu}^0(E),\;\;\;\dim_P(E)=\Dim_{\mu}^0(E)\quad\text{and}\quad\Delta(E)=\Delta_{\mu}^0(E).
$$

We define the functions $$b_{\mu}(q)=\dim_{\mu}^q(\supp\mu)
\quad\text{and}\quad B_{\mu}(q)=\Dim_{\mu}^q(\supp\mu).$$ It is well
known that the functions $b_{\mu}$ and $B_{\mu}$ are decreasing and
$B_{\mu}$ is convex and satisfying $b_{\mu}\leq B_{\mu}.$

%%%%%%%%%%%%%%%%%%%%%%%%%%%%%%%%%%%%%%%%%%%%%%%%%%%%%%%%%%%%%%%%%%%%%%%%%%%%%%
\subsection{Multifractal Hewitt-Stromberg measures and separator
functions}\label{sec2.4}
%%%%%%%%%%%%%%%%%%%%%%%%%%%%%%%%%%%%%%%%%%%%%%%%%%%%%%%%%%%%%%%%%%%%%%%%%%%%%%%%%%%

In the following, we will set up, for $q, t \in \R$ and $\mu\in
{\mathcal P}(\R^n)$, the lower and upper multifractal
Hewitt-Stromberg measures ${\mathsf{H}}_\mu^{q,t}$ and
${\mathsf{P}}_\mu^{q,t}$.\\ For $E\subseteq \supp \mu$,  the
pre-measure of $E$ is defined by
$$
 {\mathsf C}_\mu^{q,t}(E)= \limsup_{r\to0} M_{\mu,r}^q(E) (2r)^t,
$$
where
$$
M_{\mu,r}^q(E)=\sup \left\{\displaystyle \sum_i
\mu(B(x_i,r))^q\;\Big|\;\Big(B(x_i,r)\Big)_i\;\text{is a centered
packing of}\; E \right\}.
$$
It's clear that  ${\mathsf C}_\mu^{q,t}$ is increasing and ${\mathsf
C}_\mu^{q,t}(\emptyset ) =0$. However it's not $\sigma$-additive.
For this, we introduce  the ${\mathsf{P}}_\mu^{q,t}$-measure defined
by
$$
{\mathsf{P}}_\mu^{q,t}(E)=\inf \left\{\displaystyle \sum_i{\mathsf
C}_\mu^{q,t}(E_i)\;\Big|\; E\subseteq\bigcup_i E_i\; \text{and
the}\; E_i'\text{s are bounded} \right\}.
 $$
In a similar way we define
$$
{\mathsf L}_\mu^{q,t}(E)= \liminf_{r\to0} N_{\mu,r}^q(E) (2r)^t,
$$
where
$$
N_{\mu,r}^q(E)=\inf \left\{\displaystyle \sum_i
\mu(B(x_i,r))^q\;\Big|\;\Big(B(x_i,r)\Big)_i\;\text{is a centered
covering of}\; E \right\}.
$$
Since ${\mathsf L}_\mu^{q,t}$  is  not increasing and not countably
subadditive, one needs a standard modification to get an outer
measure. Hence, we modify the definition as follows
 $$
\overline{\mathsf{H}}_\mu^{q,t}(E)=\inf \left\{\displaystyle
\sum_i{\mathsf L}_\mu^{q,t}(E_i)\;\Big|\; E\subseteq\bigcup_i E_i\;
\text{and the}\; E_i'\text{s are bounded} \right\}
 $$
and
$$
\mathsf{H}_\mu^{q,t}(E)=\sup_{F\subseteq E}\overline
{\mathsf{H}}_\mu^{q,t}(F).
$$

 \bigskip \bigskip
The measure $\mathsf{ H}^{q,t}_{\mu}$ is of course a multifractal
generalization of the lower $t$-dimensional Hewitt-Stromberg measure
${\mathsf{H}}^t$, whereas $\mathsf{ P}^{q,t}_{\mu}$ is a
multifractal generalization of the upper $t$-dimensional
Hewitt-Stromberg measures ${\mathsf{P}}^t$. In fact, it is easily
seen that, for $t>0$,  one has
$$
\mathsf{ H}^{0,t}_{\mu}={\mathsf{H}}^t\quad\text{and}\quad\mathsf{
P}^{0,t}_{\mu}={\mathsf{P}}^t.
$$

The following result describes some of the basic properties of the
multifractal Hewitt-Stromberg measures including the fact that
${\mathsf H}_\mu^{q,t}$ and ${\mathsf P}_\mu^{q,t}$ are Borel metric
outer measures and summarises the basic inequalities satisfied by
the multifractal Hewitt-Stromberg measures, the generalized
Hausdorff measure and the generalized packing measure.
\begin{theorem}\label{HHPP}\cite{NB}
Let $q, t \in \R$ and  $\mu \in {\mathcal P}(\R^n)$. Then  for every
set $E\subseteq \mathbb{R}^n$ we have
\begin{enumerate}
\item the set functions $\mathsf{H}_\mu^{q,t}$ and $\mathsf{P}_\mu^{q,t}$ are
metric outer measures and thus they are measures on the Borel
algebra.
\item There exists an integer $\xi\in\mathbb{N}$, such that $$
{\mathcal H}^{q,t}_{\mu}(E)\leq \mathsf{H}_\mu^{q,t}(E)\leq \xi
\mathsf{P}_\mu^{q,t}(E)\leq \xi {\mathcal P}^{q,t}_{\mu}(E).$$
\item When $q \le 0$ or $q>0$ and $\mu \in {\mathcal P}_D (\R^n)$, we have
$${\mathcal H}^{q,t}_{\mu}(E)\leq \mathsf{H}_\mu^{q,t}(E)\leq  \mathsf{P}_\mu^{q,t}(E)\leq {\mathcal P}^{q,t}_{\mu}(E).$$
\end{enumerate}
\end{theorem}

The measures $\mathsf{ H}^{q,t}_{\mu}$ and $\mathsf{ P}^{q,t}_{\mu}$
and the pre-measure $C^{q,t}_{\mu}$ assign in the usual way a
multifractal dimension to each subset $E$ of $\mathbb{R}^n$. They
are respectively denoted by $\mathsf{b}_{\mu}^q(E)$,
$\mathsf{B}_{\mu}^q(E)$ and $\mathsf{\Delta}_{\mu}^q(E)$,
\begin{proposition} Let $q \in \R$, $\mu \in {\mathcal P}(\R^n)$ and $E\subseteq
\mathbb{R}^n$. Then
\par\noindent\begin{enumerate}
\item  there exists a unique number ${\mathsf  b}_{\mu}^q(E)\in[-\infty,+\infty]$ such that
 $$
\mathsf{H}^{q,t}_{\mu}(E)=\left\{\begin{matrix}  \infty &\text{if}& t < {\mathsf  b}_{\mu}^q(E),\\
 \\
 0 & \text{if}&  {\mathsf  b}_{\mu}^q(E) < t,\end{matrix}\right.
 $$
\item  there exists a unique number ${\mathsf  B}_{\mu}^q(E)\in[-\infty,+\infty]$ such that
 $$
\mathsf{P}^{q,t}_{\mu}(E)=\left\{\begin{matrix}  \infty &\text{if}& t < {\mathsf  B}_{\mu}^q(E),\\
 \\
 0 & \text{if}&  {\mathsf  B}_{\mu}^q(E) < t,\end{matrix}\right.
 $$

\item  there exists a unique number ${\mathsf  \Delta}_{\mu}^q(E)\in[-\infty,+\infty]$ such that
 $$
\mathsf{C}^{q,t}_{\mu}(E)=\left\{\begin{matrix}  \infty &\text{if}& t < {\mathsf  \Delta}_{\mu}^q(E),\\
 \\
 0 & \text{if}&  {\mathsf  \Delta }_{\mu}^q(E) < t.\end{matrix}\right.
 $$
\end{enumerate}
In addition, we have
$${\mathsf  b}_{\mu}^q(E) \le   {\mathsf  B }_{\mu}^q(E) \le  {\mathsf  \Delta}_{\mu}^q(E).$$
\end{proposition}
The number ${\mathsf  b}_{\mu}^q(E)$ is an obvious multifractal
analogue of the lower Hewitt-Stromberg dimension
$\underline{\dim}_{MB}(E)$ of $E$ whereas ${\mathsf  B}_{\mu}^q(E)$
is an obvious multifractal analogues of the upper Hewitt-Stromberg
dimension $\overline{\dim}_{MB}(E)$ of $E$. In fact, it follows
immediately from the definitions that
$$
{\mathsf
b}_{\mu}^0(E)=\underline{\dim}_{MB}(E)\quad\text{and}\quad{\mathsf
B}_{\mu}^0(E)=\overline{\dim}_{MB}(E).
$$
\begin{remark}
It follows from Theorem \ref{HHPP} that $$\dim_{\mu}^q(E) \le
{\mathsf b}_{\mu}^q(E) \le {\mathsf B}_{\mu}^q(E)
\le\Dim_{\mu}^q(E)\leq \Delta_{\mu}^q(E).$$
\end{remark}

 \bigskip \bigskip
The definition of these dimension functions makes it clear that they
are counterparts of the $\tau_\mu$-function which appears in the
{\it multifractal formalism}. This being the case, it is important
that they have the properties described by the physicists. The next
theorem shows that these functions do indeed have some of these
properties.
\begin{theorem}\label{th1}
Let $q\in\mathbb{R}$ and $E\subseteq \mathbb{R}^n$.
\begin{enumerate}
\item The functions $q \mapsto \mathsf{ H}^{q,t}_{\mu}(E)$,  $\mathsf{P}^{q,t}_{\mu}(E)$,  $ \mathsf{C}^{q,t}_{\mu}(E)$ are decreasing.
\item The functions $t \mapsto \mathsf{ H}^{q,t}_{\mu}(E)$,  $\mathsf{P}^{q,t}_{\mu}(E)$,  $ \mathsf{C}^{q,t}_{\mu}(E)$ are
decreasing.
\item The functions $q \mapsto { {\mathsf  b}}^{q}_{\mu}(E)$, $ {{\mathsf  B}}^{q}_{\mu}(E)$,  $ {{\mathsf  \Delta}}^{q}_{\mu}(E)$ are decreasing.
\item The functions $q \mapsto  {{\mathsf  B}}^{q}_{\mu}(E)$,  $ {{\mathsf  \Delta}}^{q}_{\mu}(E)$ are convex.
\end{enumerate}
\end{theorem}
\begin{proof}Let $q\in\mathbb{R}$ and $E\subseteq \mathbb{R}^n$.\\
The first and second part of Theorem \ref{th1} follows since $x
\mapsto a^x$ is decreasing for all $a \in ]0, 1[$. \\ Observe that
part (3) of Theorem \ref{th1} follows immediately from (1).
\\ We will now prove the part (4). Let $\alpha\in[0, 1]$ and $p, s, t \in \R$. Suppose that we
have shown that
\begin{equation}\label{convex}
\mathsf{C}_\mu^{\alpha p + (1-\alpha) q, \alpha t+ (1-\alpha )s}(E)
\le \Big( \mathsf{C}_\mu^{ p, t } (E)\Big)^\alpha  \Big (
\mathsf{C}_\mu^{ q, s } (E)\Big)^{1-\alpha}.
\end{equation}
Then, for all $\epsilon >0$, we have
\begin{equation*}
\mathsf{C}_\mu^{\alpha p + (1-\alpha) q, \alpha {{\mathsf
\Delta}}^{p}_{\mu}(E)+ (1-\alpha ){{\mathsf  \Delta}}^{q}_{\mu}(E) +
\epsilon }(E) \le \left ( \mathsf{C}_\mu^{ p, {{\mathsf
\Delta}}^{p}_{\mu}(E) +\epsilon } (E)\right)^\alpha  \left
(\mathsf{C}_\mu^{ q, {{\mathsf \Delta}}^{q}_{\mu}(E) +\epsilon  }
(E)\right)^{1-\alpha} = 0.
\end{equation*}
We therefore conclude that
$${{\mathsf  \Delta}}^{\alpha p +(1-\alpha)
q}_{\mu}(E) \le \alpha {{\mathsf  \Delta}}^{p}_{\mu}(E) +(1-\alpha)
{{\mathsf \Delta}}^{q}_{\mu}(E) + \epsilon.$$ Finally, letting
$\epsilon $ tend to $0$, then the convexity of $q \mapsto {{\mathsf
\Delta}}^{q}_{\mu}(E)$ follows.

\bigskip
We now turn towards the proof of  \eqref{convex}. Put $r >0$ and
$\Big(B(x_i, r)\Big)_i$ be a centered packing of $E$. It follows
from H\"{o}lder inequality that
 \begin{eqnarray*}
 \sum_i \mu(B(x_i, r))^{\alpha p + (1-\alpha) q} &=&  \sum_i \Big(\mu(B(x_i, r)^p\Big)^\alpha \Big(\mu(B(x_i, r)^q\Big)^{1-\alpha}\\
 &\le & \left( \sum_i \mu(B(x_i, r)^p\right)^\alpha \left(\sum_i \mu(B(x_i, r)^q\right)^{1-\alpha}\\
 &\le & \Big( M_{\mu, r}^p(E) \Big)^{\alpha}\Big( M_{\mu, r}^q(E)
\Big)^{1-\alpha}.
 \end{eqnarray*}
This shows that
$$
M_{\mu, r}^{\alpha p+ (1-\alpha)q} ~(2r)^{\alpha t+ (1-\alpha )s}\le
\Big(M_{\mu, r}^{ p} ~(2r)^{ t}\Big)^{\alpha} \Big(M_{\mu, r}^{ q}
~(2r)^{ s}\Big)^{1- \alpha}.
$$
Letting $r$ tend to  $0$ we get the result.

 \bigskip
We must now show the convexity of $q \mapsto  {{\mathsf
B}}^{q}_{\mu}(E)$.  Let $\eta >0$ and put $t=  {{\mathsf
B}}^{p}_{\mu}(E)$ and  $s= {{\mathsf B}}^{q}_{\mu}(E)$. Since
$\mathsf{P}_\mu^{q, s+\eta}(E) =\mathsf{P}_\mu^{p, t+\eta} (E)=0$,
we can choose bounded coverings $(H_i)_i$ and $(K_i)_i$ of $E$ such
that
$$
\sum_i \mathsf{C}_\mu^{q, t+\eta} (H_i)\le 1 \quad \text{and }
\sum_i \mathsf{C}_\mu^{q, s+\eta} (K_i)\le 1.
$$
Next, for $ n\in  \N^*,$ let $E_n = \displaystyle\bigcup_{i, j =1}^n
( H_i \cap K_j)$, we clearly have
\begin{eqnarray*}
\mathsf{P}_\mu^{\alpha p + (1-\alpha) q, \alpha t + (1-\alpha ) s + \eta }(E_n) &\le & \sum_{i, j= 1}^n \mathsf{P}_\mu^{\alpha p + (1-\alpha) q, \alpha t + (1-\alpha ) s + \eta } (H_i \cap K_j)\\
 &\le & \sum_{i, j= 1}^n \mathsf{C}_\mu^{\alpha p + (1-\alpha) q, \alpha t + (1-\alpha ) s + \eta } (H_i \cap K_j)\\
&\overset{\eqref{convex}}{\le} & \sum_{i, j= 1}^n \Big(
\mathsf{C}_\mu^{p, t +\eta } (H_i \cap K_j) \Big)^\alpha
\Big(  \mathsf{C}_\mu^{q, s+ \eta} (H_i \cap K_j) \Big)^{1-\alpha}\\
&\overset{\text{H\"{o}lder} }{\le} & \left(   \sum_{i, j= 1}^n \mathsf{C}_\mu^{p, t + \eta  } (H_i \cap K_j) \right)^\alpha
\left(   \sum_{i, j= 1}^n \mathsf{C}_\mu^{q, s + \eta  } (H_i \cap K_j) \right)^{1-\alpha }\\
&\le  & \left(   \sum_{i, j= 1}^n \mathsf{C}_\mu^{p, t + \eta  } (H_i ) \right)^\alpha  \left(   \sum_{i, j= 1}^n \mathsf{C}_\mu^{q, s + \eta  } ( K_j) \right)^{1-\alpha }\\
&\le  & \left( n  \sum_{i= 1}^n \mathsf{C}_\mu^{p, t + \eta  } (H_i ) \right)^\alpha  \left(n   \sum_{ j= 1}^n \mathsf{C}_\mu^{q, s + \eta  } ( K_j) \right)^{1-\alpha }\\
&\le & n^\alpha \; n^{1-\alpha} = n < \infty.
\end{eqnarray*}
We now obtain, for all $n\in \N^*$, $${{\mathsf B}}_\mu^{\alpha p +
(1-\alpha) q} (E_n)\le \alpha t + (1-\alpha) s+\eta.$$ Since clearly
$E \subseteq \bigcup_n E_n$, we therefore conclude that
\begin{eqnarray*}
{{\mathsf B}}_\mu^{\alpha p + (1-\alpha) q} (E) &\le & {{\mathsf
B}}_\mu^{\alpha p + (1-\alpha) q} \left(\bigcup_n E_n\right)
\le  \sup_n  {{\mathsf B}}_\mu^{\alpha p + (1-\alpha) q} ( E_n) \\
&\le &   \alpha  {{\mathsf B}}_\mu^p (E) + (1-\alpha)  {{\mathsf
B}}_\mu^q(E)+\eta.
\end{eqnarray*}
Letting $\eta $ tend to $0$ now yields the desired result. This
completes the proof of Theorem \ref{th1}.
\end{proof}

 \bigskip \bigskip
Next we define the multifractal separator functions ${\mathsf
b}_{\mu}$, ${\mathsf B}_{\mu}$ and ${{\mathsf \Lambda}}_{\mu}$\;:
$\mathbb{R}\rightarrow [-\infty,+\infty]$ by
$$
\begin{array}{lllcr}
{\mathsf b}_{\mu}:\;q\rightarrow {\mathsf
b}_{\mu}^{q}(\supp\mu),\quad {\mathsf B}_{\mu}:\;q\rightarrow
{\mathsf B}_{\mu}^{q}(\supp\mu)\quad\text{and}\quad {{\mathsf
\Lambda}}_{\mu}:\;q\rightarrow{{\mathsf \Delta}}_{\mu}^q(\supp\mu).
\end{array}
$$
We also obtain the following corollary providing information about
the lower and upper multifractal Hewitt-Stromberg functions.
\begin{corollary}
Let $q\in \mathbb{R}$. We have
\begin{enumerate}
\item for $q < 1$,  ${{\mathsf b}}_{\mu}(q) \ge 0$.
\item For $q = 1$,  ${{\mathsf b}}_{\mu}(q) ={{\mathsf \Lambda}}_{\mu}(q)  = 0$.
\item For $q > 1$,  ${{\mathsf \Lambda}}_{\mu}(q) \le 0$.
\end{enumerate}
\end{corollary}
\begin{proof}
This follow immediately from the above theorem and definitions.
\end{proof}

%%%%%%%%%%%%%%%%%%%%%%%%%%%%%%%%%%%%%%%%%%%%%%%%%%%%%%%%%%%%%%%%%%%%%%%%%%%%%%%%%%%%%%%%%%%%%%%%%%%%%%%%%%%%%%%%%%%
\subsection{Some characterizations of $\mathsf b_\mu(q)$ and $\mathsf
B_\mu(q)$}\label{sec2.5}
%%%%%%%%%%%%%%%%%%%%%%%%%%%%%%%%%%%%%%%%%%%%%%%%%%%%%%%%%%%%%%%%%%%%%%%%%%%%%%%%%%%%%%%%%%%%%%%%%%%%%%%%%%%%%%%%%%%%%

In this section,  we  investigate the relation between the lower and
upper multifractal Hewitt-Stromberg functions $\mathsf{b}_\mu$ and
$\mathsf{B}_\mu$ and the multifractal box dimension, the
multifractal packing dimension and the multifractal pre-packing
dimension.  We first note that there exists a unique number
${\Theta}_{\mu}^q(E)\in[-\infty,+\infty]$ such that
 $$
{\mathsf L}^{q,t}_{\mu}(E)=\left\{\begin{matrix}  \infty &\text{if}& t < {\Theta}_{\mu}^q(E),\\
 \\
 0 & \text{if}&  {\Theta}_{\mu}^q(E) < t.\end{matrix}\right.
 $$

\begin{proposition}\label{new_box}
Let $q\in\mathbb{R}$ and $\mu$ be a compact supported Borel
probability measure on $\mathbb{R}^n$. Then for every $E \subseteq
\supp \mu$ we have
$$
{\Theta}_{\mu}^q(E) = \liminf_{r\to 0} \frac{\log N_{\mu, r}^q (E) }
{-\log r}\qquad \text{and }\qquad {\mathsf \Delta}_{\mu}^q(E) =
\limsup_{r\to 0} \frac{\log M_{\mu, r}^q (E) } {-\log r}.
$$
\end{proposition}
\begin{proof}
We will prove the first equality, the second one is similar. Suppose
that $$ \displaystyle\liminf_{r\to 0} \frac{\log N_{\mu, r}^q (E) }
{-\log r} > {\Theta}_{\mu}^q(E)  +\epsilon$$ for some $\epsilon >0$.
Then we can find $\delta >0$ such that for any $r\le \delta$,
$$
N_{\mu, r}^q (E)~ r^{{  \Theta}_{\mu}^q(E)  +\epsilon} >1$$ and then
$$
\mathsf{L}_{\mu}^{q, {\Theta}_{\mu}^q(E)
+\epsilon} \ge 2^{{\Theta}_{\mu}^q(E)  +\epsilon}
$$
which is a contradiction. We therefore infer $$
\displaystyle\liminf_{r\to 0} \frac{\log N_{\mu, r}^q (E) } {-\log
r} \le {\Theta}_{\mu}^q(E) +\epsilon\;\;\text{ for any}\;\; \epsilon
>0.$$ The proof of the following statement  $$
\displaystyle\liminf_{r\to 0} \frac{\log N_{\mu, r}^q (E) } {-\log
r} \ge {\Theta}_{\mu}^q(E) - \epsilon\;\;\text{ for any}\;\;\epsilon
>0$$
is identical to the proof of the above statement and is therefore
omitted.
\end{proof}
\begin{remark} Here we follow the approach of Olsen in \cite{Ol1,
O4, Olll, Olll1}.
\begin{enumerate}
\item The multifractal dimensions ${\Theta}_{\mu}^q(E)$ and ${\mathsf
\Delta}_{\mu}^q(E)$ of $E$ represent the upper and lower
multifractal box-dimension. In particular, we have
$$
{\Theta}_{\mu}^0(E)=\underline{\dim}_{B}(E)\quad\text{and}\quad
{\mathsf \Delta}_{\mu}^0(E)=\overline{\dim}_{B}(E).
$$
\item Let us introduce the multifrcatal generalization of the
$q$-dimensions called also relative R\'{e}nyi $q$-dimensions based
on integral representations. Let $\mu$ be a probability measure on
$\mathbb{R}^n$. For $q\in\mathbb{R}\setminus\{0\}$, we write
\begin{equation*}
    \underline{D}_{\mu}^q =
    \displaystyle\liminf_{r\rightarrow0}\frac{1}{q\log r}
    \log\int \mu(B(x,r))^q d\mu(x),
\end{equation*}

and
\begin{equation*} \overline{D}_{\mu}^q =
    \displaystyle\limsup_{r\rightarrow0}\frac{1}{q\log r}
    \log\int \mu(B(x,r))^q d\mu(x).
\end{equation*}
Now we define the generalized entropies due to R\'{e}nyi by,
 $$
h^q_r(\mu)=\frac{1}{q-1} \log
M_{\mu,r}^q(\supp\mu)\quad\text{for}\quad q\neq 1
 $$
and
 $$
h^1_r(\mu)=\inf\left\{-\sum_i\mu(E_i)\log\mu(E_i)\;\Big|\; (E_i)_i
\quad\text{is a partition of}\quad \supp\mu\right\}.
 $$
We define the upper and lower R\'{e}nyi $q$-dimensions
$\overline{T}_{\mu}^q$ and $\underline{T}_{\mu}^q$ of $\mu$ by
$$ \overline{T}_{\mu}^q=\displaystyle\limsup_{r\to0}
\frac{\log h^q_r(\mu)}{ \log r} \quad \text{and}\quad
\underline{T}_{\mu}^q=\displaystyle\liminf_{r\to0} \frac{\log
h^q_r(\mu)}{\log r }.$$ If
$\underline{D}_{\mu}^q=\overline{D}_{\mu}^q$ (respectively
$\underline{T}_{\mu}^q=\overline{T}_{\mu}^q$) we refer to the common
value as the relative R\'{e}nyi $q$-dimension of $\mu$ and denote it
${D}_{\mu}^q$ (respectively ${T}_{\mu}^q$). Finally define $
\underline{\mathcal{D}}_{\mu}(q)$,
$\overline{\mathcal{D}}_{\mu}(q)$,
$\underline{\mathcal{T}}_{\mu}(q)$  and
$\overline{\mathcal{T}}_{\mu}(q)$ : $\mathbb{R}\rightarrow [-\infty,
+\infty]$ by
$$
\underline{\mathcal{D}}_{\mu}(q)=(1-q)\underline{{D}}_{\mu}^{q-1},\quad
\overline{\mathcal{D}}_{\mu}(q)=(1-q)\overline{{D}}_{\mu}^{q-1}
$$
and
$$
\underline{\mathcal{T}}_{\mu}(q)=(1-q)\underline{{T}}_{\mu}^{q},\quad
\overline{\mathcal{T}}_{\mu}(q)=(1-q)\overline{{T}}_{\mu}^{q}.
$$
Let $q\in\mathbb{R}$ and $\mu \in {\mathcal P}_D(\mathbb{R}^n)$.
Then the following holds
$$
{\mathsf \Delta}_{\mu}(q)=\underline{\mathcal{D}}_{\mu}(q) \vee
\overline{\mathcal{D}}_{\mu}(q)= \underline{\mathcal{T}}_{\mu}(q)
\vee \overline{\mathcal{T}}_{\mu}(q).
$$
\item We define the multifractal Minkowski volume as follows. Let $E$ be a subset of $\mathbb{R}^n$ and $r > 0$.
We denote by $B(E, r)$ the open $r$ neighbourhood of $E$, i.e.
$$
B(E, r)=\Big\{ x\in \mathbb{R}^n \;\Big|\; \text{dist}(x,E)<r
\Big\}.
$$
For a real number $q$ and a Borel measure $\mu$ on $\mathbb{R}^n$,
we define the multifractal  Minkowski volume $V_{\mu,r}^q (E)$ of
$E$ with respect to the measure $\mu$ by
$$
V_{\mu,r}^q (E)= \frac{1}{r^n} \int_{B(E, r)} \mu(B(x,r))^q d
\mathcal{L}^n(x).
$$
Here  $\mathcal{L}^n$ denotes the $n$-dimensional Lebesgue measure
in $\mathbb{R}^n$. The importance of the R\'{e}nyi dimensions in
multifractal analysis together with the formal resemblance between
the multifractal Minkowski volume $V_{\mu,r}^q (E)$ and the moments
$\int_E \mu(B(x,r))^{q-1} d\mu(x) $ used in the definition the
R\'{e}nyi dimensions may be seen as a justification for calling the
quantity $V_{\mu,r}^q (E)$ for the multifractal Minkowski volume.
Using the multifractal Minkowski volume we can define multifractal
Monkowski dimensions. For a real number $q$ and a Borel measure
$\mu$ on $\mathbb{R}^n$, we define the lower and upper multifractal
Minkowski dimension of $E$, by
$$
\underline{\dim}_{M,\mu}^q(E)=\liminf_{r\to0}\frac{\log V_{\mu,r}^q
(E)}{-\log r} \quad\text{and}\quad
\overline{\dim}_{M,\mu}^q(E)=\limsup_{r\to0}\frac{\log V_{\mu,r}^q
(E)}{-\log r}.
$$
We note the close similarity between the multifractal Minkowski
dimensions and ${\mathsf \Delta}_{\mu}^q$.  Indeed, the equality
\eqref{nbnb} shows that this similarity is not merely a formal
resemblance. In fact, for $q \geq 1$, the multifractal Minkowski
dimensions and ${\mathsf \Delta}_{\mu}^q$ coincide, i.e. for
$q\geq1$ and $\mu \in {\mathcal P}_D(\mathbb{R}^n)$, we have
\begin{eqnarray}\label{nbnb}
{\mathsf \Delta}_{\mu}(q)=\underline{\dim}_{M,\mu}^{q-1}(\supp\mu)
\vee \overline{\dim}_{M,\mu}^{q-1}(\supp\mu).
\end{eqnarray}
\end{enumerate}
\end{remark}
\begin{proposition}\label{mod_box}
Let $q\in\mathbb{R}$ and $\mu$ be a compact supported Borel
probability measure on $\mathbb{R}^n$. Then for every $E \subseteq
\supp \mu$ we have
$$
{\mathsf  b}_{\mu}^q(E) =\sup_{F\subseteq E} \left\{ \inf \left\{
\sup_{i} {\Theta}_{\mu}^q(F_i) \;\Big|\;\displaystyle F\subseteq
\bigcup_i F_i, \;\; F_i \;\; \text{are bounded in } \;\;
\R^n\right\} \right\}
$$
and
$$
{\mathsf  B}_{\mu}^q(E) = \inf \left\{ \sup_{i}  {\mathsf
\Delta}_{\mu}^q (E_i) \;\Big|\;\displaystyle E\subseteq \bigcup_i
E_i, \;\; E_i \;\; \text{are bounded in } \;\; \R^n\right\}.
$$

\end{proposition}
\begin{proof}
Denote
$$\beta = \sup_{F\subset E} \left\{ \inf \left\{ \sup_{i}
{\Theta}_{\mu}^q(F_i)\;\Big|\; F\subseteq \bigcup_i F_i, \;\; F_i
\;\; \text{are bounded in } \;\; \R^n\right\}  \right\}.$$ Assume
that $\beta < {\mathsf  b}_{\mu}^q(E)$ and take $\alpha \in (\beta,
{\mathsf b}_{\mu}^q(E) )$. Then,  for all $F \subseteq E$, there
exists $\{F_i\}$ of bounded subset of $F$
 such that  $F\subseteq \cup_i F_i$, and $\sup_i {\Theta}_{\mu}^q(F_i) < \alpha$.
Now observe that $\mathsf{L}_{\mu}^{q, \alpha}(F_i) = 0$ which
implies that ${\overline{\mathsf H}}_{\mu}^{q, \alpha}(F) = 0$.
 This implies that ${\mathsf H}_{\mu}^{q, \alpha}(E) = 0$. It is a contradiction. Now suppose that
$ {\mathsf  b}_{\mu}^q(E) < \beta$, then, for $\alpha \in ( {\mathsf
b}_{\mu}^q(E), \beta ),$ we have ${\mathsf H}_{\mu}^{q, \alpha}(E) =
0$. It follows from this that  ${\overline {\mathsf H}}_{\mu}^{q,
\alpha}(F) = 0$ for all $F\subseteq E$. Thus, there exists $\{F_i\}$
of bounded subset of $F$ such that  $F\subseteq \cup_i F_i$, and
$\sup_i {\mathsf L}_{\mu}^{q, \alpha} (F_i) < \infty$. We conclude
that, $\sup_i {\Theta}_{\mu}^{q} (F_i) \le \alpha$. It is also a
contradiction.

The proof of the second statement is identical to the proof of the
statement in the first part and is therefore omitted.
\end{proof}

\begin{proposition}\label{ourBB}
If $q\in\mathbb{R}$ and $\mu \in {\mathcal P}_D(\mathbb{R}^n)$, then
for any subset $E $ of $\supp \mu$, we have
$${\mathsf  B}_{\mu}^q(E) = { B}_{\mu}^q(E). $$
\end{proposition}
\begin{proof}
 This follows easily from Propositions \ref{new_box} and \ref{mod_box},
Propositions 2.19 and 2.22 in \cite{Ol1} and Lemma 4.1 in \cite{O1}.
\end{proof}

\begin{remark}
The results developed by Falconer in \cite{FKJ} are obtained as a
special case of the multifractal results by setting $q=0$.
\end{remark}

%%%%%%%%%%%%%%%%%%%%%%%%%%%%%%%%%%%%%%%%%%%%%%%%%%%%%%%%%%%%%%%%%%%%%%%%%%%%%%%%%%%%%%%%%%%%%%%%%%%%%%%%%%%%%%%%%%%%%%%%%%%%%%%%%%%%%%%%%%
                                                                                     \section{A multifractal formalism for Hewitt-Stromberg
                                                                                     measures}\label{sec3}
%%%%%%%%%%%%%%%%%%%%%%%%%%%%%%%%%%%%%%%%%%%%%%%%%%%%%%%%%%%%%%%%%%%%%%%%%%%%%%%%%%%%%%%%%%%%%%%%%%%%%%%%%%%%%%%%%%%%%%%%%%%%%%%%%%%%%%%%%%

Multifractal analysis was proved to be a very useful technique in
the analysis of measures, both in theory and applications. The
upper and lower local dimensions of a measure $\mu$ on
$\mathbb{R}^n$ at a point $x$ are respectively given by :
$$
 \overline{\alpha}_\mu(x) = \limsup_{r\to 0}\frac{\log \mu(B(x,r))}{\log r} \quad \text{and} \quad \underline{\alpha}_\mu(x) = \liminf_{r\to 0}\frac{\log \mu(B(x,r))}{\log r},
$$
where $B(x,r)$ denote the closed ball of center $x$ and radius $r$. We refer to the common value as the local dimension of $\mu$ at
$x$, and denote it by ${\alpha}_{\mu}(x)$.\\

The level set of the local dimension of $\mu $  contains a crucial
information on the geometrical properties of $\mu$. The aim of the
multifractal analysis of a measure  is to relate the Hausdorff and
packing dimensions of these levels sets to the Legendre  transform
of some concave (convex) function (see for example \cite{Naj, BBJ,
BJ, BB, BBH, DB, TO1, Ol1}). For $\alpha \ge 0$, we define the
fractal sets,
$$
{\overline E}^\alpha =  \Big\{ x \in \supp \mu
\;\;\big|\;\;\overline{\alpha}_{\mu}(x)   \le \alpha \Big\}; \quad
\overline{E}_\alpha =  \Big\{ x \in \supp \mu
\;\;\big|\;\;\overline{\alpha}_{\mu }(x)   \ge \alpha \Big\}
$$
and
$$
{\underline E}^\alpha =  \Big\{ x \in \supp \mu
\;\;\big|\;\;\underline{\alpha}_{\mu}(x)   \le \alpha \Big\}; \quad
\underline{E}_\alpha =  \Big\{ x \in \supp \mu
\;\;\big|\;\;\underline{\alpha}_{\mu }(x)   \ge \alpha \Big\}.
$$
Also, let
$$
E(\alpha) = {\overline E}^\alpha \cap {\underline E}_\alpha = \Big\{
x \in \supp \mu \;\;\big|\;\;\alpha_{\mu}(x)   =  \alpha \Big\}.
$$

 \bigskip \bigskip
Theorem \ref{dDmeasure} allows us to consider the relationship
between the lower and upper multifractal Hewitt-Stromberg functions
$\mathsf{b}_\mu$ and $\mathsf{B}_\mu$ and the multifractal spectra.
We start by giving an upper bound theorem. For $\mu \in {\mathcal
P}(\R^n)$, set
$$
\alpha_{min} = \sup_{0<q} -\frac{{{\mathsf b}_\mu(q)}}{q}, \quad
\alpha_{max} = \inf_{0>q} -\frac{{{\mathsf b}_\mu(q)}}{q},\quad
\beta_{min} = \sup_{0<q} -\frac{{{\mathsf B}_\mu(q)}}{q}, \quad
\beta_{max} = \inf_{0>q} -\frac{{{\mathsf B}_\mu(q)}}{q}.
$$
Before stating this formally, we remind the reader that if $\varphi
: \mathbb{R} \rightarrow  \mathbb{R}$ is a real valued function,
then the Legendre transform $\varphi^* : \mathbb{R} \rightarrow
[-\infty, +\infty]$ of $\varphi$ is defined by
$$
\varphi^*(x)=\inf_y\Big(xy+ \varphi(y)\Big).
$$

Now, we can state our {\it multifractal formalism}.
\begin{theorem}\label{dDmeasure}
Let $\alpha \ge 0$, then the following hold\\
\begin{enumerate}
\item $$\alpha_{min}  \le \inf \overline{\alpha}_{\mu} (x) \le \sup \overline{\alpha}_{\mu}(x) \le \beta_{max}$$ and
$$\beta_{min}  \le \inf \underline{\alpha}_{\mu} (x) \le \sup \underline{\alpha}_{\mu}(x) \le
\alpha_{max}.$$
\item $$ \underline{\dim}_{MB} \; E(\alpha)
\begin{cases}
\le {\mathsf b}_\mu^*(\alpha) & \;\text{if}\;\;\alpha \in (\alpha_{min}, \alpha_{max})\\
&\\
=0 &\;\text{if}\;\;\alpha \notin (\alpha_{min}, \alpha_{max}).
\end{cases}
$$
\item   $$ \overline{\dim}_{MB} \; E(\alpha)
\begin{cases}
\le {\mathsf B}_\mu^*(\alpha) & \;\text{if}\;\;\alpha \in (\alpha_{min}, \alpha_{max})\\
&\\
=0 & \;\text{if}\;\;\alpha \notin (\alpha_{min}, \alpha_{max}).
\end{cases}
$$
\end{enumerate}
\end{theorem}
\begin{proof}
 This theorem follows immediately from the following lemmas.
\end{proof}
\begin{lemma}
If  $\mu \in {\mathcal P}(\R^n)$ and $\alpha \ge 0$, then\\
\begin{enumerate}
\item ${\underline E}^\alpha  = \emptyset  $\;\;  for\;\;  $\alpha < \beta_{min}$ \qquad\, and \qquad ${\overline E}_\alpha = \emptyset $ \;\; for \;\; $\alpha > \beta_{max}$.\\
\item ${\underline E}_\alpha  =  \emptyset  $\;\;  for \;\; $\alpha > \alpha_{max}$ \qquad and \qquad ${\overline E}^\alpha = \emptyset $\;\;  for\;\;  $\alpha < \alpha_{min}$.
\end{enumerate}
\end{lemma}
\begin{proof}
\begin{enumerate}
\item Let   $ x \in {\underline E}^\alpha $ and $\alpha <
\beta_{min}$. There exists $\epsilon >0$ and $q_0 >0$ such that
$-q_0(\alpha +\epsilon)> {\mathsf B}_\mu(q_0)$. Since  $ x \in
{\underline E}^\alpha $, we can thus choose $ (r_n)_n$ such that
$$
r_n \to 0 \quad \text{and} \quad \frac{\log\mu(B(x, r_n))}{\log r_n}
< \alpha + \epsilon.
$$
For brevity write $t =- q_0 (\alpha+ \epsilon)$, then we obtain
$$\mu(B(x, r_n))^{q_0}~ (2r_n)^t > 2^t.$$ Hence, for all $n\in \N$,
$$
M_{\mu, r_n}^{q_0, t}\big(\{x\}\big) ~(2 r_n)^t \ge \mu(B(x,
r_n))^{q_0}~ (2 r_n)^t
>  2^t >0.
$$
It follows that  ${\mathsf P}_\mu^{q_0, t} \big(\{x\}\big) >0$. We
therefore conclude that $$t \le  {\mathsf B}_\mu^{q_0}
\big(\{x\}\big) \le {\mathsf B}_\mu(q_0) $$ which contradicts the
fact that $-q_0(\alpha +\epsilon)> {\mathsf B}_\mu(q_0)$.

The proof of the second statement is identical to the proof of the
first statement in Part (1) and is therefore omitted.
\item Let   $ x \in {\underline E}_\alpha $ and $\alpha >
\alpha_{max}$. Then, we can find $\epsilon >0$ and $q_0 < 0$ such
that $-q_0(\alpha -\epsilon)> {\mathsf b}_\mu(q_0)$. Since  $ x \in
{\underline E}_\alpha $,  we can choose $r_0$ such that for $0 < r <
r_0$ we have
$$
\frac{\log\mu(B(x, r))}{\log r} > \alpha - \epsilon.
$$
Then, for $t =- q_0 (\alpha- \epsilon)$ and $r< r_0$,  we have
$$\mu(B(x, r))^{q_0}~ (2r)^t > 2^t.$$ Therefore, it follows that, for
all $r < r_0$,
$$
N_{\mu, r}^{q_0, t}\big(\{x\}\big) ~(2 r)^t  = \mu(B(x, r))^{q_0} (2
r )^t > 2^t >0.
$$
which implies that  ${\mathsf H}_\mu^{q_0, t} \big(\{x\}\big) >0$.
It now follows from this that $$t \le  {\mathsf b}_\mu^{q_0}
\big(\{x\}\big) \le  {\mathsf b}_\mu(q_0) $$ which contradicts the
fact that  $-q_0(\alpha - \epsilon)> {\mathsf b}_\mu(q_0)$.

The proof of the second statement in part (2) is very similar to the
proof of the first statement and is therefore omitted.
\end{enumerate}
\end{proof}
\begin{lemma}
Let $\mu \in {\mathcal P}(\R^n)$, $\alpha \ge 0, q, t\in \R$  and
$\delta>0$ such that $\delta \le \alpha q + t.$ Then the following hold \\
\begin{enumerate}
\item
\begin{enumerate}
\item  $ { \mathsf{H}}^{\alpha q  + t +\delta}  ({\overline E}^\alpha) \le \; 2^{\alpha q+\delta}~\mathsf{H}_\mu^{q, t} ( {\overline E}^\alpha) \;\; \text{for}\;\; \; 0\le q.$
\item $ { \mathsf{H}}^{\alpha q  + t +\delta}  ({\underline E}_\alpha) \le \; 2^{\alpha q+\delta}~\mathsf{H}_\mu^{q, t} ( {\underline E}_\alpha) \;\; \text{for}\;\; \; 0\ge q.$
\item If $0\le \alpha q +  {{\mathsf b}}_{\mu}(q)$ then
$$
\underline{\dim}_{MB} ( {\overline E}^\alpha)  \le \inf_{q \ge 0}
\alpha q + {{\mathsf b}}_{\mu}(q)  \quad \text{and}\quad
\underline{\dim}_{MB} ( {\underline E}_\alpha)  \le\inf_{q \le 0}
\alpha q + {{\mathsf b}}_{\mu}(q).
$$
In particular  $\underline{\dim}_{MB} ( {\overline E}^\alpha)  \le \alpha$.\\
\end{enumerate}
\item
\begin{enumerate}
\item  $ { \mathsf{P}}^{\alpha q  + t +\delta}  ({\overline E}^\alpha) \le \; 2^{\alpha q+\delta}~\mathsf{P}_\mu^{q, t} ( {\overline E}^\alpha) \;\; \text{for}\;\; \; 0\le q.$
\item $ { \mathsf{P}}^{\alpha q  + t +\delta}  ({\underline E}_\alpha) \le \; 2^{\alpha q+\delta}~\mathsf{P}_\mu^{q, t} ( {\underline E}_\alpha) \;\; \text{for}\;\; \; 0\ge q.$
\item If $0\le \alpha q +  {{\mathsf b}}_{\mu}(q)$ then
$$
\overline{\dim}_{MB} ( {\overline E}^\alpha)  \le \inf_{q \ge 0}
\alpha q + {{\mathsf B}}_{\mu}(q)  \quad \text{and}\quad
\overline{\dim}_{MB} ( {\underline E}_\alpha)  \le\inf_{q \le 0}
\alpha q + {{\mathsf B}}_{\mu}(q).
$$
In particular  $\overline{\dim}_{MB} ( {\overline E}^\alpha)  \le \alpha$.\\
\end{enumerate}

\end{enumerate}
\end{lemma}
\begin{proof} An exhaustive proof of this lemma would require considerable repetition. To avoid this we prove (1)-(a) and (2)-(a).
\begin{enumerate}
\item \begin{enumerate}\item
Clearly, the statement is true for $q= 0$. For $m\in \N$, write
$$
E_m = \left\{ x \in {\overline E}^\alpha\;\;\Big|\;\;\frac{\log
\mu(B(x,r))}{\log r} \le \alpha +  \frac{\delta}{q} \;\; \text{for}
\;\;0< r < \frac{1}{m} \right\}.
$$
  Fix $ m\in \N$ and $r >0$ such that $0< r < \frac{1}{m}$. Let $\Big(B(x_i, r)\Big)_i$ be a centered covering of $E_m$. Next, we observe that
\begin{eqnarray*}
\frac{\log \mu(B(x_i,r))}{\log r} \le \alpha +  \frac{\delta}{q}&\Longrightarrow & \mu(B(x_i,r))^q  \ge r^{q\alpha +  \delta} \\
&\Longrightarrow&  \sum_i  \mu(B(x_i,r))^q  \ge N_r(E_m)~  r^{q\alpha +  \delta} \\
&\Longrightarrow &  N^q_{\mu, r}  (E_m)  \ge N_r(E_m)~  r^{q\alpha +  \delta}\\
&\Longrightarrow  &  \mathsf{L}_{\mu}^{q,t}  (E_m)  \ge 2^{-\alpha q-\delta}~ \overline{\mathsf{H}}^{q\alpha +  \delta + t}(E_m)\\
&\Longrightarrow &  \mathsf{ H}_{\mu}^{q,t}  (E_m)  \ge
\overline{\mathsf{H}}_{\mu}^{q,t}  (E_m)
\ge 2^{-\alpha q-\delta} ~\overline{\mathsf{H}}^{q\alpha +  \delta + t}(E_m). \\
\end{eqnarray*}
Now from this and since  $E_m \nearrow {\overline E}^\alpha$  we can
deduce that $$ \mathsf{H}_{\mu}^{q,t} ({\overline E}^\alpha)   \ge
2^{-\alpha q-\delta}~ \mathsf{H}^{q\alpha + \delta + t} ({\overline
E}^\alpha).$$
\end{enumerate}
\item \begin{enumerate}\item
Once again for $q = 0$ the statement is well known. For $m\in \N$,
put
$$
E_m = \left\{ x \in {\overline E}^\alpha\;\;\Big|\;\; \frac{\log
\mu(B(x,r))}{\log r} \le \alpha +  \frac{\delta}{q} \;\; \text{for}
\;\;0< r < \frac{1}{m} \right\}.
$$
We therefore fix $ m\in \N$ and $r >0$ such that $0< r <
\frac{1}{m}$. Let $\Big( B(x_i, r)\Big)_{i\in \{1, \ldots,
M_r(E_m)\}}$ be a packing of $ E_m$. Then we have
\begin{eqnarray*}
\frac{\log \mu(B(x_i,r))}{\log r} \le \alpha +   \frac{\delta}{q}&\Longrightarrow & \mu(B(x_i,r))^q  \ge r^{q\alpha +  \delta} \\
&\Longrightarrow&  \sum_i  \mu(B(x_i,r))^q  \ge M_r(E_m) ~ r^{q\alpha +  \delta} \\
&\Longrightarrow &  M^q_{\mu, r}  (E_m)  \ge M_r(E_m) ~ r^{q\alpha + \delta}\\
&\Longrightarrow  &  \mathsf{C}_{\mu}^{q,t}  (E_m)  \ge2^{-\alpha q-\delta}~ \overline{\mathsf{P}}^{q\alpha + \delta + t}(E_m)\\
&\Longrightarrow &  \mathsf{P}_{\mu}^{q,t}  (E_m)  \ge 2^{-\alpha q-\delta}~ \mathsf{P}^{q\alpha +  \delta + t}(E_m). \\
\end{eqnarray*}
However, since $E_m \nearrow {\overline E}^\alpha, $ we conclude
that $$ \mathsf{P}_{\mu}^{q,t} ({\overline E}^\alpha)   \ge
2^{-\alpha q-\delta}~ \mathsf{P}^{q\alpha + \delta + t} ({\overline
E}^\alpha)$$  which yields the desired result.
\end{enumerate}
\end{enumerate}
\end{proof}

Our purpose of the following theorems is  to propose a sufficient
condition that gives the lower bound.
\begin{theorem}\label{thmbBbB}
Let $q, t \in \R$, and $\alpha >0$ such that $\alpha q  + t \ge 0$.
Let  $A \subseteq E(\alpha)$ is  a Borel set.\\
\begin{enumerate}
\item If  $\mathsf{H}_\mu^{q, t} (A) > 0$, then
$$\underline{\dim}_{MB} \; E(\alpha) \ge \alpha q + t.$$
 In particular, if the multifractal function ${\mathsf b}_\mu$ is differentiable at $q,$ then, provided that
  $ {\mathsf b}_\mu^*\Big( - {\mathsf b}_\mu^{'}(q)\Big) \ge 0$  and
$\mathsf{H}_\mu^{q, {\mathsf b}_\mu(q)} \Big( E(- {\mathsf b_\mu}'
(q)) \Big) > 0$, we have
 $$
  \underline{\dim}_{MB}  E\left( - {\mathsf b}_\mu^{'}(q)\right) =  {\mathsf b}_\mu^*\left( - {\mathsf b}_\mu^{'}(q)\right).
 $$
\item If  $\mathsf{P}_\mu^{q, t} (A) > 0$, then
$$\overline{\dim}_{MB} \; E(\alpha) \ge \alpha q + t.$$
 In particular, if the multifractal function ${\mathsf B}_\mu$ is differentiable at $q,$ then, provided that
 $ {\mathsf B}_\mu^*\left( - {\mathsf B}_\mu^{'}(q)\right) \ge 0$  and
$\mathsf{P}_\mu^{q, {\mathsf B}_\mu(q)} \Big( E(- {\mathsf B_\mu}'
(q)) \Big) > 0$, we have
 $$
  \overline{\dim}_{MB}  E\left( - {\mathsf B}_\mu^{'}(q)\right) =  {\mathsf B}_\mu^*\left( - {\mathsf B}_\mu^{'}(q)\right).
 $$

 \end{enumerate}
\end{theorem}
\begin{proof} This follows easily from Theorem \ref{dDmeasure} and the following
lemma.
\end{proof}
\begin{lemma}\label{LEMMA3}
Let $\mu \in {\mathcal P}(\R^n)$, $\alpha \ge 0, q, t\in \R$  and
$\delta>0$ such that $\delta \le \alpha q + t.$ Then we have the following\\
\begin{enumerate}
\item
\begin{enumerate}
\item  If $A\subseteq {\overline E}^\alpha,$ is Borel then $ { \mathsf{P}}^{\alpha q  + t -\delta}  (A ) \ge \; 2^{\alpha q-\delta}~\mathsf{P}_\mu^{q, t} ( A ) \;\; \text{for}\;\; \; 0\ge q.$
\item If $A\subseteq {\underline E}_\alpha$, is Borel then   $ { \mathsf{P}}^{\alpha q  + t -\delta}  (A)
\ge \; 2^{\alpha q-\delta}~\mathsf{P}_\mu^{q, t} ( A )\;\; \text{for}\;\; \; 0\le q.$ \\In particular, if $\mu(A) >0$ then    $\overline{\dim}_{MB} ( A )  \ge \alpha$.\\
\end{enumerate}
\item
\begin{enumerate}
\item  If $A\subseteq {\overline E}^\alpha,$ is Borel then $ { \mathsf{H}}^{\alpha q  + t -\delta}  (A ) \ge \; 2^{\alpha q-\delta}~\mathsf{H}_\mu^{q, t} ( A ) \;\; \text{for}\;\; \; 0\ge q.$
\item If $A\subseteq {\underline E}_\alpha$, is Borel then   $ { \mathsf{H}}^{\alpha q  + t -\delta}  (A)
\ge  \; 2^{\alpha q-\delta}~\mathsf{H}_\mu^{q, t} ( A )\;\; \text{for}\;\; \; 0\le q.$ \\In particular, if $\mu(A) >0$ then    $\underline{\dim}_{MB} ( A )  \ge \alpha$.\\
\end{enumerate}

\end{enumerate}
\end{lemma}
\begin{proof}
An exhaustive proof of this theorem would require considerable
repetition. For this we only prove (1)-(a) and (2)-(a), the other
assertions are similar.

\begin{enumerate}

\item
\begin{enumerate}
\item Clearly the statement is true for $q= 0$. For $m\in \N$, write
$$
E_m = \left\{ x \in A \;\;\Big|\;\; \frac{\log \mu(B(x,r))}{\log r}
\le \alpha -  \frac{\delta}{q} \;\; \text{for} \;\;0< r <
\frac{1}{m} \right\}.
$$
Let $ m\in \N$ and $r >0$ such that $0< r < \frac{1}{m}$. Let $\Big(
B(x_i, r)\Big)_i$ be a centred packing  of $E_m$. We have
\begin{eqnarray*}
\frac{\log \mu(B(x_i,r))}{\log r} \le \alpha -  \frac{\delta}{q}&\Longrightarrow & \mu(B(x_i,r))^q  \le r^{q\alpha -  \delta} \\
&\Longrightarrow&  \sum_i  \mu(B(x_i,r))^q  \le M_r(E_m)~  r^{q\alpha -  \delta} \\
&\Longrightarrow &  M_{\mu, r}^q  (E_m)  \le M_r(E_m)~  r^{q\alpha -  \delta}\\
&\Longrightarrow  &  \mathsf{C}_{\mu}  (E_m)^{q, t}  \le2^{-\alpha q+\delta}~ \overline{\mathsf{P}}^{q\alpha -  \delta + t}(E_m)\\
&\Longrightarrow &   \mathsf{P}_{\mu}^{q,t}  (E_m)  \le 2^{-\alpha q+\delta}~ \mathsf{P}^{q\alpha -  \delta + t}(E_m). \\
\end{eqnarray*}
Finally, since $E_m \nearrow A $ we conclude that $$
\mathsf{P}_{\mu}^{q,t} (A) \le 2^{-\alpha q+\delta}~
\mathsf{P}^{q\alpha -  \delta + t} (A).$$

\end{enumerate}
\item \begin{enumerate}\item
It is well known that the statement is true for $q= 0$. For $m\in
\N$, we define the set $E_m$ by
$$
E_m = \left\{ x \in A\;\;\Big|\;\; \frac{\log \mu(B(x,r))}{\log r}
\le \alpha -  \frac{\delta}{q} \;\; \text{for} \;\;0< r <
\frac{1}{m} \right\}.
$$
Next, fix $ m\in \N$ and $r >0$ such that $0< r < \frac{1}{m}$. Let
$\Big( B(x_i, r)\Big)_{i\in \{1, \ldots, N_r(F)\}}$ be a centred
covering of $F \subset E_m$. We get
\begin{eqnarray*}
\frac{\log \mu(B(x_i,r))}{\log r} \le \alpha -  \frac{\delta}{q}&\Longrightarrow & \mu(B(x_i,r))^q  \le r^{q\alpha -  \delta} \\
&\Longrightarrow&  \sum_i  \mu(B(x_i,r))^q  \le N_r(F)~  r^{q\alpha -  \delta} \\
&\Longrightarrow &  N^q_{\mu, r}  (F)  \le N_r(F)~  r^{q\alpha - \delta}\\
&\Longrightarrow  &  \mathsf{L}_{\mu}^{q,t}  (F)  \le 2^{t-\alpha q+\delta} ~\overline{\mathsf{H}}^{q\alpha - \delta + t}(F)\\
&\Longrightarrow &  \overline{ \mathsf{H}}_{\mu}^{q,t}  (F)
\le 2^{-\alpha q+\delta}~ \overline{\mathsf{H}}^{q\alpha -  \delta + t}(E_m). \\
\end{eqnarray*}
Putting these together we have that $$ \mathsf{H}_{\mu}^{q,t} (A)
\le 2^{-\alpha q+\delta} ~\mathsf{H}^{q\alpha -  \delta + t} (A).$$
This proves the lemma.
\end{enumerate}
\end{enumerate}
\end{proof}

\begin{theorem}\label{formalisme}
Let $q\in\mathbb{R}$ and suppose that $\mathsf{ H}^{q, \mathsf
\Lambda_{\mu}(q)}_{\mu}(\supp\mu)>0.$ Then,\\
\begin{eqnarray*}
 \begin{array}{ll}
  \underline{\dim}_{MB} \Big(\underline{E}_{\;-\mathsf\Lambda_{\mu+}'(q)}\cap\overline{E}^{\;-\mathsf\Lambda
_{\mu-}'(q)}\Big)
\end{array}
\geq\left\{
\begin{array}{ll}
-\mathsf\Lambda_{\mu-}'(q)q+\mathsf
\Lambda_{\mu}(q),\quad\text{for}\quad q\leq0 \hbox{,} \\
\\-\mathsf\Lambda_{\mu+}'(q)q+ \mathsf\Lambda_{\mu}(q),\quad
\text{for}\quad q\geq0 \hbox{.}
\end{array}
\right.
\end{eqnarray*}
\end{theorem}
\begin{proof} It is well known from Lemma \ref{LEMMA3} that for all
$\delta>0$ and $t\in \mathbb{R}$,
\begin{eqnarray*}
\left\{ \begin{array}{ll} \mathsf{H}^{-\mathsf\Lambda'_{\mu+}(q)q+t-\delta}(E(q))\geq2^{-\mathsf\Lambda'_{\mu+}(q) q-\delta}~ \mathsf{H}^{q,t}_{\mu}(E(q)) ,\quad\text{for}\quad q\geq0 \hbox{,} \\
&\\
\mathsf{H}^{-\mathsf\Lambda'_{\mu-}(q)q+t-\delta}(E(q))\geq2^{-\mathsf\Lambda'_{\mu-}(q)
q-\delta}~\mathsf{H}^{q,t}_{\mu}(E(q)),\quad\text{for}\quad q\leq0
\hbox{}
\end{array}
\right.
\end{eqnarray*}
where the set $E(q)$ is defined by
$$E(q)=\underline{E}_{\;-\mathsf\Lambda_{\mu+}'(q)}\cap\overline{E}^{\;-\mathsf\Lambda_{\mu-}'(q)}.$$
Theorem \ref{formalisme} is then an easy consequence of the
following lemma. \end{proof}
\begin{lemma}\label{lem} One has $\mathsf{H}^{q,\mathsf\Lambda_{\mu}(q)}_{\mu}\Big(\supp\mu\setminus E(q) \Big)=0.$
\end{lemma}
\begin{proof}Let us introduce, for $\alpha$ and $\beta$ in $\mathbb{R}$
$$
X_\alpha=\supp\mu\setminus\underline{E}_\alpha\quad \text{and}\quad
Y^\beta=\supp\mu\setminus\overline{E}^\beta.
$$
It clearly suffices to prove that
\begin{eqnarray}\label{eq1}
\mathsf{H}^{q,\mathsf\Lambda_{\mu}(q)}_{\mu}\big(X_\alpha\big)=0,\quad\text{for
all}\quad \alpha<-\mathsf{\Lambda}_{\mu+}'(q)
\end{eqnarray}
and
\begin{eqnarray}\label{eq2}
\mathsf{ H}^{q,\mathsf
\Lambda_{\mu}(q)}_{\mu}\big(Y^\beta\big)=0,\quad\text{for all}\quad
\beta>-\mathsf\Lambda_{\mu-}'(q).
\end{eqnarray}
Indeed, it is clear that
\begin{eqnarray*}
0 &\leq& \mathsf{H}^{q,
\mathsf\Lambda_{\mu}(q)}_{\mu}\Big(\supp\mu\setminus\big(\underline{E}_{\;-\mathsf\Lambda_{\mu+}'(q)}\cap\overline{E}^{\;-\mathsf\Lambda
_{\mu-}'(q)}\big) \Big)\\
&\leq& \mathsf{
H}^{q,\mathsf\Lambda_{\mu}(q)}_{\mu}\Big(\supp\mu\setminus\underline{E}_{\;-\mathsf\Lambda_{\mu+}'(q)}\Big)
+ \mathsf{H}^{q,\mathsf\Lambda_{\mu}(q)}_{\mu}\Big(\supp\mu \setminus\overline{E}^{\;-\mathsf\Lambda_{\mu-}'(q)}\Big)\\
&\leq&  \mathsf{ H}^{q,\mathsf\Lambda_{\mu}(q)}_{\mu}\left(
\bigcup_{\alpha<-\mathsf\Lambda_{\mu+}'(q)}\underline{E}_\alpha\right)
+ \mathsf{H}^{q,\mathsf\Lambda_{\mu}(q)}_{\mu}\left(\bigcup_{\beta>- \mathsf\Lambda_{\mu-}'(q)}\overline{E}^\beta\right)\\
&\leq& \sum_{\alpha<-\mathsf\Lambda_{\mu+}'(q)}\mathsf{
H}^{q,\mathsf\Lambda_{\mu}(q)}_{\mu}\big(X_\alpha\big)+\sum_{\beta>
-\mathsf\Lambda_{\mu-}'(q)}\mathsf{
H}^{q,\mathsf\Lambda_{\mu}(q)}_{\mu}\big(Y^\beta\big)=0.
\end{eqnarray*}

We only have to prove that (\ref{eq1}). The proof of (\ref{eq2}) is
identical to the proof of (\ref{eq1}) and is therefore omitted.

\bigskip
Let $\alpha<-\mathsf \Lambda_{\mu+}'(q)$ and  $t>0,$ such that
$\mathsf\Lambda_{\mu}(q+t)< \mathsf\Lambda_{\mu}(q)-\alpha t$, we
have
$$
\mathsf{C}^{q+t, \mathsf\Lambda_{\mu}(q)-\alpha
t}_{\mu}\big(\supp\mu\big)=0.
$$
For $x\in X_{\alpha}$ and $\delta>0$, we can find $\lambda_x \ge 2$
and $\frac{\delta}{\lambda_x}<r_x<\delta,$ such that
 $$
    \mu(B(x,r_x))> r_x^\alpha.
 $$
The family $\Big(B(x,r_x)\Big)_{x\in X_{\alpha}}$ is a centered
$\delta$-covering of $ \overline{X}_{\alpha}.$ Then, we can choose a
finite subset $J$ of $\N$ such that the family
$\Big(B(x_i,r_{x_i})\Big)_{i\in J}$ is a centered $\delta$-covering
of $X_\alpha$. Take $\lambda = \sup_{i\in J} \lambda_{x_i}$, then
for all $i\in J$, we have
$$
 \mu(B(x_i, \delta)) \ge  \mu(B(x_i,r_{x_i}))> r_{x_i}^\alpha \ge \left(\frac{\delta}{\lambda}\right)^\alpha.
$$

  Since $\Big(B(x_i , \delta )\Big)_{ i\in J}$  is a centered covering of $ \overline{X}_{\alpha}.$ Then, using Besicovitch's covering
theorem, we can construct $\xi$ finite  sub-families $\Big(B(x_{1j},
\delta)\Big)_j$, \ldots ,$\Big(B(x_{\xi j}, \delta )\Big)_j$, such
that each  $
X_{\alpha}\subseteq\displaystyle\bigcup_{i=1}^\xi\bigcup_jB(x_{ij},
\delta)$ and $\Big(B(x_{ij}, \delta)\Big)_j$ is a packing of
$X_{\alpha}$. We clearly have
 $$
\mu(B(x_{ij}, \delta)^q ~  \delta^{\mathsf\Lambda_{\mu}(q)}\leq
\lambda^{\alpha t} ~  \mu(B(x_{ij}, \delta)^{q+t}
~\delta^{\mathsf\Lambda_{\mu}(q)-\alpha t} .
 $$
It therefore follows that
 $$
N_\mu^{q} (X_\alpha) ~\delta^{\mathsf\Lambda_{\mu}(q)}\leq
\lambda^{\alpha t} ~\xi ~ M_\mu^{q+t}(X_\alpha)
~\delta^{\mathsf\Lambda_{\mu}(q)-\alpha t} .
$$
Letting $\delta\to 0$  now yields
 $$
\overline {\mathsf{H}}_{\mu}^{q,
\mathsf\Lambda_{\mu}(q)}(X_{\alpha}) \leq \mathsf{L}_{\mu}^{q,
\mathsf\Lambda_{\mu}(q)}(X_{\alpha}) \leq 2^{\alpha
t}~\lambda^{\alpha t}~ \xi~ \mathsf{C}_\mu^{q+t,
\mathsf\Lambda_{\mu}(q)-\alpha t}(X_\alpha) \leq 2^{\alpha
t}~\lambda^{\alpha t} ~\xi  ~\mathsf{C}_\mu^{q+t,
\mathsf\Lambda_{\mu}(q)-\alpha t}( \supp \mu) = 0.
$$
Remark that, in the last inequality, we can replace $X_{\alpha}$ by
any arbitrary subset of $X_{\alpha}.$ Then, we can finally conclude
that
$$ {\mathsf{H}}_{\mu}^{q,
\mathsf\Lambda_{\mu}(q)}(X_{\alpha}) \leq 2^{\alpha
t}~\lambda^{\alpha t} ~\xi  ~\mathsf{C}_\mu^{q+t,
\mathsf\Lambda_{\mu}(q)-\alpha t}( \supp \mu) = 0.
$$
This completes the proof of \eqref{eq1}. \end{proof}

  \bigskip
The following result proves that the condition $
\mathsf{H}^{q,\mathsf \Lambda_{\mu}(q)}_{\mu}(\supp\mu)>0 $ is very
close to being a necessary and sufficient condition for the validity
of our {\it multifractal formalism}.

\begin{theorem}\label{SNC}
Let $q\in\mathbb{R}$ and $\mu$ be a compact supported Borel
probability measure on $\mathbb{R}^n.$ Suppose that one of the
following hypotheses is satisfied,
\begin{enumerate}
\item
$\underline{\dim}_{MB}
 \Big(\underline{E}_{\;-\mathsf\Lambda_{\mu+}'(q)}\cap\overline{E}^{\;-\mathsf\Lambda
_{\mu-}'(q)}\Big)\geq-\mathsf\Lambda_{\mu+}'(q)q+
\mathsf\Lambda_\mu(q),\quad\text{for}\quad q\leq0.$
\\
\item$ \underline{\dim}_{MB}  \Big(\underline{E}_{\;-\mathsf\Lambda_{\mu+}'(q)}
\cap\overline{E}^{\;-\mathsf\Lambda_{\mu-}'(q)}\Big)\geq-\mathsf\Lambda_{\mu-}'(q)q+\mathsf\Lambda_\mu(q),\quad\text{for}\quad
q\geq0.$
\end{enumerate}
Then,
 $$
\mathsf b_{\mu}(q)=\mathsf B_{\mu}(q)=\mathsf\Lambda_{\mu}(q).
 $$
In other words,
$$
 {\mathsf{H}}^{q,t}_{\mu}(\supp\mu)>0\quad\text{for all}\quad t<\mathsf
\Lambda_{\mu}(q).
$$
\end{theorem}
\begin{proof} We have, for $q\geq0$
$$
\underline{E}_{\;-\mathsf\Lambda_{\mu+}'(q)}\cap\overline{E}^{\;-\mathsf\Lambda_{\mu-}'(q)}\subseteq
\overline{E}^{\;-\mathsf\Lambda_{\mu-}'(q)},
$$
it follows immediately that
$$
-\mathsf\Lambda_{\mu-}'(q)q+ \mathsf\Lambda_\mu(q) \leq
\underline{\dim}_{MB}\Big(\underline{E}_{\;-\mathsf\Lambda_{\mu+}'(q)}\cap\overline{E}^{\;-\mathsf\Lambda_{\mu-}'(q)}\Big)\leq
\underline{\dim}_{MB}
\Big(\overline{E}^{\;-\mathsf\Lambda_{\mu-}'(q)}\Big).
$$
Now, suppose that $\alpha=-\mathsf\Lambda_{\mu-}'(q)$. We only prove
the case where $q \geq0$. The other one is very similar and is
therefore omitted. We have
$$
\underline{\dim}_{MB} \Big(\overline{E}^\alpha\Big)\geq \alpha
q+\mathsf\Lambda_\mu(q).
$$
Since $\mathsf b_{\mu}(q)\leq \mathsf B_{\mu}(q)\leq
\mathsf\Lambda_{\mu}(q)$, we only have  to prove that $\mathsf
b_{\mu}(q)\geq \mathsf\Lambda_{\mu}(q)$. Let $t<
\mathsf\Lambda_{\mu}(q)$ and choose $\beta,$ such that
$\beta<\alpha$. Then, $\beta q+t<\alpha q+\mathsf\Lambda_{\mu}(q)$.
For $p\in \mathbb{N},$ we consider the set
$$
F_p=\left\{x\in\overline{E}^\alpha\;\;\Big|\;\;\mu(B(x,r))\geq
r^\beta,\; 0<r<\frac1p\right\}.
$$
It is clear that $F_p\nearrow \overline{E}^\alpha$ as $p\to \infty$.
It follows that, there exists $p>0$, such that
$$
\underline{\dim}_{MB} (F_p)>\beta q+t \Rightarrow \mathsf{H}^{\beta
q+t}(F_p)>0.
$$
Let $0<r <\frac1p$ and $\Big(B(x_i,r)\Big)_i$ be a centered covering
of $F_p$. Then,
$$
\sum_i\mu(B(x_i,r))^q ~r^{t}\geq\sum_i r^{\beta q+t}\ge N_r(F_p )~
r^{\beta q+t}.
$$
We conclude that
$$
N_{\mu,r}^q(F_p )~ (2r)^{t}\ge 2^t ~N_r(F_p )~ r^{\beta q+t}$$ and
then $$ \mathsf{L}_{\mu}^{q,t} (F_p)\ge 2^{-\beta q }~
\overline{\mathsf{H}}^{\beta q+t}(F_p ).
$$
This implies that
$$
\mathsf{H}^{q,t}_{\mu}(\supp\mu)\geq\mathsf{H}^{q,t}_{\mu}(\overline{E}^\alpha)\geq\mathsf{H}^{q,t}_{\mu}(F_p)\geq2^{-\beta
q }~\mathsf{H}^{\beta q+t}(F_p)>0.
$$
It therefore follows that $t\leq \mathsf  b_{\mu}(q)$. Finally, we
get
$$
\mathsf  b_{\mu}(q)=\mathsf  B_{\mu }(q)=\mathsf \Lambda_{\mu}(q).
$$
\end{proof}
\begin{corollary}\label{co1}
Assume that $\mathsf{ H}^{q, \mathsf
\Lambda_{\mu}(q)}_{\mu}(\supp\mu)>0$ hold for all $q\in\mathbb{R}$
and that $\mathsf\Lambda_{\mu}$ is differentiable at $q$. Let
$\alpha=-\mathsf\Lambda_{\mu}'(q)$, there holds
$$
\underline{\dim}_{MB}\big(E(\alpha)\big)=\overline{\dim}_{MB}\big(E(\alpha)\big)
= \mathsf  b^*_{\mu}(\alpha)=\mathsf  B^*_{\mu}(\alpha)=\mathsf
\Lambda^*_{\mu}(\alpha).
 $$
\end{corollary}

\begin{remark}
The results of Theorems \ref{formalisme}, \ref{SNC} and Corollary
\ref{co1} hold if we replace the multifractal function
$\mathsf\Lambda_{\mu}$ by the function $\mathsf B_{\mu}$.
\end{remark}

 \bigskip \bigskip
Now, we deal with the case where the lower and upper multifractal
Hewitt-Stromberg functions $\mathsf{b}_\mu$ and $\mathsf{B}_\mu$ do
not necessarily coincide.
\begin{theorem} \label{newform} Let $q\in\mathbb{R}$ and $\mu$ be a compact supported Borel
probability measure on $\mathbb{R}^n.$
\begin{enumerate}
\item
If the multifractal function ${\mathsf b}_\mu$ is differentiable at
$q,$ then, provided that
  $ {\mathsf b}_\mu^*\Big( - {\mathsf b}_\mu^{'}(q)\Big) \ge 0$ \\ and
${\mathcal H}_\mu^{q, {\mathsf b}_\mu(q)} \Big( E\left( - {\mathsf
b}_\mu^{'}(q)\right) \Big)  > 0$, we have
 $$
  \dim_HE\left( - {\mathsf b}_\mu^{'}(q)\right)=\underline{\dim}_{MB}  E\left( - {\mathsf b}_\mu^{'}(q)\right) = { b}_\mu^*\left( - {\mathsf b}_\mu^{'}(q)\right)= {\mathsf b}_\mu^*\left( - {\mathsf b}_\mu^{'}(q)\right).
 $$
\item
If the multifractal function ${ B}_\mu$ is differentiable at $q,$
then, provided that
  $ {B}_\mu^*\Big( - {B}_\mu^{'}(q)\Big) \ge 0$ \\ and
${\mathsf P}_\mu^{q, {B}_\mu(q)} \Big( E\left( -
{B}_\mu^{'}(q)\right) \Big) > 0$, we have
 $$
  \dim_P E\left( - {B}_\mu^{'}(q)\right)=\overline{\dim}_{MB}  E\left( - {B}_\mu^{'}(q)\right)
 = {\mathsf B}_\mu^*\left( - {B}_\mu^{'}(q)\right)= {B}_\mu^*\left( - {B}_\mu^{'}(q)\right).
 $$
\end{enumerate}
\end{theorem}
\begin{proof}
The proof is similar to the one of Theorem \ref{thmbBbB}.
\end{proof}

%%%%%%%%%%%%%%%%%%%%%%%%%%%%%%%%%%%%%%%%%%%%%%%%%%%%%%%%%%%%%%%%%%%%%%%%%%%%%%%%%%%%%%%%%%%%%%%%%%%%%%%%%%%%%%%%%%%%%%%%%%%%%%%%%%%%%%%
                                                                \section{Examples}\label{sec4}
%%%%%%%%%%%%%%%%%%%%%%%%%%%%%%%%%%%%%%%%%%%%%%%%%%%%%%%%%%%%%%%%%%%%%%%%%%%%%%%%%%%%%%%%%%%%%%%%%%%%%%%%%%%%%%%%%%%%%%%%%%%%%%%%%%%

In this section, more motivations and examples related to these
concepts, will be discussed.

%%%%%%%%%%%%%%%%%%%%%%%%%%%%%%%%%%%%%%%%%%%%%%%%%%%%%%%%%%%%%%%%%%%%%%%%%%%%%%%%%%%%%%%%%%%%%%%%%%%%%%%%
\subsection{Example 1}
%%%%%%%%%%%%%%%%%%%%%%%%%%%%%%%%%%%%%%%%%%%%%%%%%%%%%%%%%%%%%%%%%%%%%%%%%%%%%%%%%%%%%%%%%%%%%%%%%%%%%%%%%

The classical multifractal formalism  has been proved rigorously for
random and non-random self-similar measures \cite{Ol1, O3}, for
self-affine measures \cite{BBJ, O4}, for quasi self-similar measures
\cite{TO1}, for quasi-Bernoulli measures \cite{BBJ}, for graph
directed self-conformal measures \cite{Ol1} and for some Moran
measures \cite{W, W1}. Specifically, we have
$$
{ b}_{\mu}(q)={\mathsf  b}_{\mu}(q) = {\mathsf  B}_{\mu}(q) = {
B}_{\mu}(q)
$$
and for some $\alpha\geq0$, we get
$$
{f}_\mu(\alpha)=\mathsf{f}_\mu(\alpha)=
\mathsf{F}_\mu(\alpha)={F}_\mu(\alpha)={b}_\mu^*(\alpha)=\mathsf{b}_\mu^*(\alpha)=\mathsf{B}_\mu^*(\alpha)={B}_\mu^*(\alpha).
$$

%%%%%%%%%%%%%%%%%%%%%%%%%%%%%%%%%%%%%%%%%%%%%%%%%%%%%%%%%%%%%%%%%%%%%%%%%%%%%%%%%%%%%%%%%%%%%%%
\subsection{Example 2 : Multifractal formalism of homogeneous Moran measures}
%%%%%%%%%%%%%%%%%%%%%%%%%%%%%%%%%%%%%%%%%%%%%%%%%%%%%%%%%%%%%%%%%%%%%%%%%%%%%%%%%%%%%%%%%%%%%%%%

We will start by defining the homogeneous Moran sets. Let
$\{n_k\}_k$ and $\{\Phi_k\}_{k\geq1}$ be respectively two sequences
of positive integers and positive vectors such that
 $$
\Phi_k=\Big(c_{k_1}, c_{k_2}, \ldots, c_{k_{n_k}}\Big), \qquad
\displaystyle\sum_{j=1}^{n_k}c_{kj}\leq 1, \; k\in \N.
 $$
For any $m,k\in \N$, such that $m\leq k$, let
 $$
D_{m,k}=\Big\{(i_m, i_{m+1}, \ldots, i_k)\;\;\big|\;\; 1\leq i_j\leq
n_j, m\leq j\leq k\Big\}
 $$
and
 $$
D_k=D_{1,k}=\Big\{(i_1, i_{2}, \ldots, i_k)\;\;\big|\;\; 1\leq
i_j\leq n_j, 1\leq j\leq k\Big\}.
 $$
We also set
 $$
D_0=\emptyset \qquad \text{and}\qquad D=\displaystyle\cup_{k\geq
0}D_k,
 $$
Considering $\sigma =(i_1, i_{2}, \ldots, i_k)\in D_k$, $\tau=(j_1,
j_{2}, \ldots, j_m)\in D_{k+1,m}$, we set
 $$
\sigma *\tau =(i_1, i_{2}, \ldots, i_k, j_1, j_{2}, \ldots, j_m).
 $$

\begin{definition}
Let $X$ be a complete metric space and $I\subseteq X$ a compact set
with no empty interior (for convenience, we assume that the diameter
of $I$ is 1). The collection ${\mathcal F} =\{I_\sigma, \sigma\in
D\}$ of subsets of $I$ is said to have a homogeneous Moran
structure, if it satisfies the following conditions (MSC):
\begin{description}
        \item[a] $I_\emptyset=I$.
\item[b] For all $k\geq 1$, $(i_1, i_{2}, \ldots, i_{k-1})\in D_{k-1}$,
$I_{i_1i_{2}\ldots i_k} \big(i_k\in \{1, 2, \ldots, n_k\}\big)$ are
subsets of $I_{i_1i_{2}\ldots i_{k-1}}$ and
 $$
I^\circ_{i_1i_{2}\ldots i_{k-1},i_k}\cap I^\circ_{i_1i_{2}\ldots
i_{k-1}, i_k'}=\emptyset, \qquad 1\leq i_k<i_k'\leq n_k,
 $$
where $I^\circ$ denotes the interior of $I$.
\item[c] For all $k\geq 1$ and $1\leq j\leq n_k$, taking $(i_1, i_{2},
\ldots, i_{k-1}, j)\in D_k$, we have
 $$
0<c_{kj}=c_{i_1i_{2} \ldots i_{k-1}j}
=\displaystyle\frac{|I_{i_1i_{2}\ldots
i_{k-1}j}|}{|I_{i_1i_{2}\ldots i_{k-1}}|}<1, \quad k\geq 2,
 $$
where $|I|$ denotes the diameter of $I$.
\end{description}
\end{definition}

 \bigskip \bigskip
Suppose that ${\mathcal F}$ is a collection of subsets of $I$ having
a homogeneous Moran structure. We call $
E=\displaystyle{\bigcap_{k\geq 1}\bigcup_{\sigma\in D_k}} I_\sigma,$
a homogeneous Moran set determined by ${\mathcal F}$, and call $
{\mathcal F}_k=\Big\{\sigma\;\;\big|\;\; \sigma\in D_k\Big\}$ the
$k$-order fundamental sets of $E$. $I$ is called the original set of
$E$. We assume $ \displaystyle\lim_{k\to \infty} \max_{\sigma\in
D_k}|I_\sigma| =0.$ Then, for all $i\in D$, the set
$\left\{\displaystyle\bigcap_{n\geq1} I_{i_1i_{2}\ldots
i_n}\right\}$ is a single point.  We use the abbreviation $w|_k$ for
the first $k$ elements of the sequence
 $$
w =(i_1, i_{2}, \ldots, i_k, \ldots)\in D, \qquad I_k(w) =I_{w|_k}
=I_{i_1i_{2}\ldots i_k}.
 $$

 \bigskip
\bigskip
Here, we consider a class of homogeneous Moran sets $E$ witch
satisfy a special property called the strong separation condition
(\texttt{SSC}), i.e., take $I_\sigma\in {\mathcal F}$. Let
$I_{\sigma
*1}, I_{\sigma
*2}, \ldots, I_{\sigma *n_{k+1}}$ be the $n_{k+1}$ basic intervals of order
$k + 1$ contained in $I_\sigma$ arranged from the left to the right,
Then we assume that for all $1\leq i\leq n_{k+1}-1$,
 $$
\text{dist} (I_{\sigma *i}, I_{\sigma *(i+1)}) \geq \delta_k
|I_\sigma|, \quad \text{ for all } i\neq j,
 $$
where $(\delta_k)_k$ is a sequence of positive real numbers, such
that
 $$
0<\delta=\displaystyle\inf_k \delta_k.
 $$

We now define a Moran measure. Let $\Big\{p_{i,j}\Big\}_{j=1}^{n_i}$
be the probability vectors, i.e. $p_{i,j}>0$ and
$\sum_{j=1}^{n_i}p_{i,j}=1$ ($i=1,2,3,....$), suppose that $p_0=\inf
\{p_{i,j}\}>0$.  Let $\mu$ be a mass distribution on $E$, such that
for any $I_\sigma$ ($\sigma\in D_k$)
$$
\mu(I_\sigma)=p_{1,\sigma_1}p_{1,\sigma_2}\ldots
p_{1,\sigma_k}\quad\text{and}\quad \mu\left(\sum_{\sigma\in D_k}
I_\sigma \right)=1,
$$
we call $\mu$ be Moran measure.

 \bigskip \bigskip
Finally we define an auxiliary function $\beta_k(q)$ as follows: for
all $k \geq 1$ and $q\in \mathbb{R}$, there is a unique number
$\beta_k(q)$ satisfying
$$
\sum_{\sigma\in D_k} p_\sigma^q |I_\sigma|^{\beta_k(q)}=1.
$$
Set
$$
\underline{\beta}(q)=\liminf_{k\to+\infty}\beta_k(q)\quad\text{and}\quad\overline{\beta}(q)=\limsup_{k\to+\infty}\beta_k(q)
$$
\begin{theorem}\label{forvalid}
Suppose that $E$ is a homogeneous Moran set satisfying
(\texttt{SSC}) and $\mu$ is the Moran measure on $E$,
\begin{enumerate}
\item
then for all $q \in\mathbb{R}$,
$$
{ b}_{\mu}(q)={\mathsf
b}_{\mu}(q)={\Theta}_{\mu}^q(\supp\mu)=\underline{\beta}(q)
$$
and
$$
{ B}_{\mu}(q)={\mathsf B}_{\mu}(q)={{\mathsf
\Lambda}}_{\mu}^q(\supp\mu)=\overline{\beta}(q).
$$
\\
\item  Suppose that $\underline{\beta}'(q)$ exists and for this
real number $q$

\begin{enumerate}
\item  there is $k_0\in\mathbb{N}$ such that
$\underline{\beta}(q)\leq\beta_k(q)$ for all $k\geq k_0$, or
\item there is some $c > 0$ and $n_0\in\mathbb{N}$ such that
$\beta_{k_i}(q)-\underline{\beta}(q)\leq \frac{c}{k_i}$ for all
$k_i\geq n_0$ with $\beta_{k_i}(q)<\underline{\beta}(q)$.
\item $\underline{\beta}(q)$ is smooth.
\end{enumerate}

\noindent Then there exist numbers
$0\leq\underline{\alpha}\leq\overline{\alpha}$ such that
$$
{f}_\mu(\alpha)=\mathsf{f}_\mu(\alpha)= \left\{
                                          \begin{array}{ll}
                                            b_{\mu}^*(\alpha)={\mathsf
b}_{\mu}^*(\alpha)=\underline{\beta}^*(\alpha), & \hbox{if}
\;\alpha\in(\underline{\alpha},\overline{\alpha})\\ \\
                                            0, & \hbox{if}\;
\alpha\notin(\underline{\alpha},\overline{\alpha}).
                                          \end{array}
                                        \right.
$$
\\
\item  Suppose that $\overline{\beta}'(q)$ exists and for this
real number $q$

\begin{enumerate}
\item  there is $k_0\in\mathbb{N}$ such that
$\overline{\beta}(q)\geq\beta_k(q)$ for all $k\geq k_0$, or
\item there is some $c > 0$ and $n_0\in\mathbb{N}$ such that
$\overline{\beta}(q)-\beta_{k_i}(q)\leq \frac{c}{k_i}$ for all
$k_i\geq n_0$ with $\beta_{k_i}(q)<\overline{\beta}(q)$.
\item $\overline{\beta}(q)$ is smooth.
\end{enumerate}

\noindent Then there exist numbers
$0\leq\underline{\gamma}\leq\overline{\gamma}$ such that
$$
{F}_\mu(\alpha)=\mathsf{F}_\mu(\alpha)= \left\{
                                          \begin{array}{ll}
                                            B_{\mu}^*(\alpha)={\mathsf
B}_{\mu}^*(\alpha)=\overline{\beta}^*(\alpha), & \hbox{if}
\;\alpha\in(\underline{\gamma},\overline{\gamma})\\ \\
                                            0, & \hbox{if}\;
\alpha\notin(\underline{\gamma},\overline{\gamma}).
                                          \end{array}
                                        \right.
$$
\\
\item If the limit $\liminf_{k\to+\infty}\beta_k(q) = \beta(q)$ exists,
and for all $k\geq 1$, $k(\beta(q)-\beta_k(q))<+\infty$, suppose
that $\alpha=-\beta'(q)$ exists, then

\begin{eqnarray*}\label{fFbB}
{f}_\mu(\alpha)=\mathsf{f}_\mu(\alpha)=
\mathsf{F}_\mu(\alpha)={F}_\mu(\alpha)={b}_\mu^*(\alpha)=\mathsf{b}_\mu^*(\alpha)
=\mathsf{B}_\mu^*(\alpha)={B}_\mu^*(\alpha).
\end{eqnarray*}
\end{enumerate}
\end{theorem}
\begin{proof}
All of the ideas needed to prove Theorem \ref{forvalid} can be found
in \cite{W3, W4}, Propositions \ref{new_box}, \ref{mod_box} and
\ref{ourBB} and Theorem \ref{newform}.
\end{proof}

%%%%%%%%%%%%%%%%%%%%%%%%%%%%%%%%%%%%%%%%%%%%%%%%%%%%%%%%%%%%%%%%%%%%%%%%%%%%%%%%%%%%%%%%%%%%%%%%%%%%%%%%%%%%%%%%%ù
\subsubsection{{\bf Moran measures for which the classical multifractal formalism is valid}}
%%%%%%%%%%%%%%%%%%%%%%%%%%%%%%%%%%%%%%%%%%%%%%%%%%%%%%%%%%%%%%%%%%%%%%%%%%%%%%%%%%%%%%%%%%%%%%%%%%%%%%%%%%%%%%%%%%%%%%

Let \begin{eqnarray*}
                n_k&=& \left\{
                         \begin{array}{ll}
                           2, & \hbox{k\;\text{ is odd number},} \\
\\
                           3, & \hbox{k\;\text{ is even number}.}
                         \end{array}
                       \right.
                \end{eqnarray*}
\begin{eqnarray*}
                c_k&=& \left\{
                         \begin{array}{ll}
                           r_1, & \hbox{k\;\text{ is odd number},}
\\ \\
                           r_2, & \hbox{k\;\text{ is even number,}}
                         \end{array}
                       \right.
                \end{eqnarray*}
where $0<r_1<\frac{1}2$ and  $0<r_2<\frac{1}3$.  Put
\begin{eqnarray*}
                p_{k,j}&=& \left\{
                         \begin{array}{ll}
                           p_{1,j}, & \hbox{k\;\text{ is odd number},\; } 1\leq j\leq
2,\\ \\
                           p_{2,j}, & \hbox{k\;\text{ is even number,\; }}1\leq j\leq 3,
                         \end{array}
                       \right.
                \end{eqnarray*}
where
$$
\sum_{j=1}^2p_{1,j}=1\quad\text{and}\quad\sum_{j=1}^3p_{2,j}=1.
$$
We therefore conclude that
\begin{eqnarray*}
                \beta_{k}(q)&=& \left\{
                         \begin{array}{ll}
                           \displaystyle\frac{-\log\sum_{j=1}^2p_{1,j}^q-\frac{k-1}{k+1}
\log\sum_{j=1}^3p_{2,j}^q}{\log r_1+\frac{k-1}{k+1}\log r_2}, &
\hbox{k\;\text{ is odd number}, } \\ \\ \\
                           \displaystyle\frac{-\log\sum_{j=1}^2p_{1,j}^q-
\log\sum_{j=1}^3p_{2,j}^q}{\log r_1+\log r_2}, & \hbox{k\;\text{ is
even number, }}
                         \end{array}
                       \right.
                \end{eqnarray*}
and
$$
\beta(q)=\lim_{k\to+\infty}\beta_{k}(q)=
\displaystyle\frac{-\log\sum_{j=1}^2p_{1,j}^q-
\log\sum_{j=1}^3p_{2,j}^q}{\log r_1+\log r_2}.
$$
This clearly implies that $k(\beta(q)-\beta_k(q))<+\infty$ and
$\beta'(q)$ exists. Now, it follows immediately from Theorem
\ref{forvalid} that
$$
{f}_\mu(\alpha)=\mathsf{f}_\mu(\alpha)=
\mathsf{F}_\mu(\alpha)={F}_\mu(\alpha)={b}_\mu^*(\alpha)=\mathsf{b}_\mu^*(\alpha)
=\mathsf{B}_\mu^*(\alpha)={B}_\mu^*(\alpha).
$$

%%%%%%%%%%%%%%%%%%%%%%%%%%%%%%%%%%%%%%%%%%%%%%%%%%%%%%%%%%%%%%%%%%%%%%%%%%%%%%%%%%%%%%%%%%%%%%%%%%%%%%%%%%%%%%%%%%%%%%%%%%%%%ù
\subsubsection{{\bf Moran measures for which the classical multifractal formalism does not hold}}
%%%%%%%%%%%%%%%%%%%%%%%%%%%%%%%%%%%%%%%%%%%%%%%%%%%%%%%%%%%%%%%%%%%%%%%%%%%%%%%%%%%%%%%%%%%%%%%%%%%%%%%%%%%%%%%%%%%%%%%%%%%%%%%%%%

Let $\{T_k\}_{k\geq1}$ be a sequence of integers such that
$$
T_1=1,\quad
T_k<T_{k+1}\quad\text{and}\quad\lim_{k\to+\infty}\frac{T_{k+1}}{T_{k}}=+\infty
$$
We define the family
\begin{eqnarray*}
                n_i&=& \left\{
                         \begin{array}{ll}
                           2, & \hbox{\text{if}\;}  T_{2k-1}\leq
i<T_{2k},\\ \\
                           3, & \hbox{\text{if}\;} T_{2k}\leq i<T_{2k+1}.
                         \end{array}
                       \right.
                \end{eqnarray*}
\begin{eqnarray*}
                c_i&=& \left\{
                         \begin{array}{ll}
                           r_1, & \hbox{\text{if}\;} T_{2k-1}\leq
i<T_{2k},\\ \\
                           r_2, & \hbox{\text{if}\;}T_{2k}\leq i<T_{2k+1},
                         \end{array}
                       \right.
                \end{eqnarray*}
where $0<r_1<\frac{1}2$ and  $0<r_2<\frac{1}3$.  Put
\begin{eqnarray*}
                p_{i,j}&=& \left\{
                         \begin{array}{ll}
                           p_{1,j}, & \hbox{\text{if}\;  } T_{2k-1}\leq i<T_{2k},\;1\leq j\leq
2,\\ \\
                           p_{2,j}, & \hbox{\text{if}\; }T_{2k}\leq i<T_{2k+1},\;1\leq j\leq 3,
                         \end{array}
                       \right.
                \end{eqnarray*}
where
$$
\sum_{j=1}^2p_{1,j}=1\quad\text{and}\quad\sum_{j=1}^3p_{2,j}=1.
$$
We therefore conclude from this
$$
\beta_k(q)=\frac{\log\sum_{\sigma\in D_k} \mu(I_\sigma)^q }{-\log
c_1 \ldots c_k}.
$$
Finally, if $N_k$ is the number of integers $i\leq k$  such that
$p_{i,j}=p_{1,j}$, we have
$$
\beta_k(q)=-\frac{\frac{N_k}{k}\log(p_{1,1}^q+p_{1,2}^q)+(1-\frac{N_k}{k})\log(p_{2,1}^q+p_{2,2}^q+p_{2,3}^q)}{\frac{N_k}{k}\log
r_1+(1-\frac{N_k}{k})\log r_2}.
$$
Observing that
$$
\liminf_{k\to+\infty}\frac{N_k}{k}=0\quad\text{and}\quad
\limsup_{k\to+\infty}\frac{N_k}{k}=1.
$$
We can then conclude that
$$
\liminf_{k\to+\infty}\beta_k(q)=\inf\left\{\frac{\log(p_{1,1}^q+p_{1,2}^q)}{-\log
r_1},\;\; \frac{\log(p_{2,1}^q+p_{2,2}^q+p_{2,3}^q)}{-\log
r_2}\right\}
$$
and
$$
\limsup_{k\to+\infty}\beta_k(q)=\sup\left\{\frac{\log(p_{1,1}^q+p_{1,2}^q)}{-\log
r_1},\;\; \frac{\log(p_{2,1}^q+p_{2,2}^q+p_{2,3}^q)}{-\log
r_2}\right\}.
$$
It results that for $0 < q < 1$, we have
$$
{ b}_{\mu}(q)={\mathsf
b}_{\mu}(q)=\underline{\beta}(q)=\liminf_{k\to+\infty}\beta_k(q)=\frac{\log(p_{1,1}^q+p_{1,2}^q)}{-\log
r_1}<
$$
$$
{ B}_{\mu}(q)={\mathsf
B}_{\mu}(q)=\overline{\beta}(q)=\limsup_{k\to+\infty}\beta_k(q)=
\frac{\log(p_{2,1}^q+p_{2,2}^q+p_{2,3}^q)}{-\log r_2}
$$
and, for $q < 0$ or $q > 1$,
$$
{ b}_{\mu}(q)={\mathsf
b}_{\mu}(q)=\underline{\beta}(q)=\liminf_{k\to+\infty}\beta_k(q)=\frac{\log(p_{2,1}^q+p_{2,2}^q+p_{2,3}^q)}{-\log
r_2}<
$$
$$
{ B}_{\mu}(q)={\mathsf
B}_{\mu}(q)=\overline{\beta}(q)=\limsup_{k\to+\infty}\beta_k(q)=\frac{\log(p_{1,1}^q+p_{1,2}^q)}{-\log
r_1}.
$$

%%%%%%%%%%%%%%%%%%%%%%%%%%%%%%%%%%%%%%%%%%%%%%%%%%%%%%%%%%%%%%%%%%%%%%%%%%%%%%%%%%%%%%%%%%%%%%%%%%%%%%%%ù
\subsection{Example 3}
%%%%%%%%%%%%%%%%%%%%%%%%%%%%%%%%%%%%%%%%%%%%%%%%%%%%%%%%%%%%%%%%%%%%%%%%%%%%%%%%%%%%%%%%%%%%%%%%%%%%%%%%%ù

In the following, we give an example of a measure for which the
lower and upper multifractal Hewitt-Stromberg functions are
different and the Hausdorff and packing dimensions of the level sets
of the local H\"{o}lder exponent $E(\alpha)$ are given by the
Legendre transform respectively of lower and upper multifractal
Hewitt-Stromberg functions.  Take $0<p<\hat{p}\leq 1/2$ and a
sequence of integers
$$
1=t_0<t_1<\ldots<t_n<\ldots, \;\; \text{such that}
\;\;\lim_{n\to+\infty}\frac{t_{n+1}}{t_n}=+\infty. $$ The measure
$\mu$ assigned to the diadic interval of the n-th generation
$I_{\varepsilon_1\varepsilon_2\ldots\varepsilon_n}$ is
$$
\mu\big(I_{\varepsilon_1\varepsilon_2\ldots\varepsilon_n}\big)=\prod_{j=1}^n
\varpi_j,
$$
where
$$
\left\{
  \begin{array}{ll}
  \text{if}\;\; t_{2k-1}\leq j<t_{2k}\;\;\text{for some}\;\;k,\;
\varpi_j=p\;\text{if}\;\varepsilon_j=0,\; \varpi_j=1-p\;\;\text{otherwise}   , &\\
\\
  \text{if}\;\; t_{2k}\leq j<t_{2k+1}\;\;\text{for some}\;\;k,\;
 \varpi_j=\hat{p}\;\text{if}\;\varepsilon_j=0,\; \varpi_j=1-\hat{p}\;\;\text{otherwise.} &
  \end{array}
\right.
$$
Now, for $q\in \mathbb{R}$ we define,

$$
\tau(q)=\log_2\big(p^q+(1-p)^q\big)\quad\text{and}\quad
\hat{\tau}(q)=\log_2\big(\hat{p}^q+(1-\hat{p})^q\big).
$$
It results from \cite{BJ, BBH} that

$$
\left\{
  \begin{array}{ll}
    b_{\mu}(q)={\mathsf
b}_{\mu}(q)=\tau(q)<\hat{\tau}(q)={\mathsf B}_{\mu}(q)=B_{\mu}(q),
\;\; \text{for}\;\; 0<q<1,&
\\   \\
b_{\mu}(q)={\mathsf b}_{\mu}(q)=\hat{\tau}(q)<\tau(q)={\mathsf
B}_{\mu}(q)=B_{\mu}(q), \;\; \text{for}\;\; q<0\;\text{or}\; q>1.
  \end{array}
\right.
$$
Then we have the following result,
\begin{theorem}Let $\alpha\geq0$.
\begin{enumerate}
\item  For $\alpha\in
\Big(-\log_2(1-\hat{p}),\;-\log_2(\hat{p})\Big)$, then we have
$$
{f}_\mu(\alpha)=\mathsf{f}_\mu(\alpha)= b_{\mu}^*(\alpha)={\mathsf
b}_{\mu}^*(\alpha).
$$
\\
\item  For $\alpha\in
\Big(-\log_2(1-\hat{p}),\;-\log_2(\hat{p})\Big)\setminus\Big(
\Big[-B_{\mu_+}'(0),-B_{\mu_-}'(0)\Big]\bigcup
\Big[-B_{\mu_+}'(1),-B_{\mu_-}'(1)\Big]\Big)$, we have
$$
{F}_\mu(\alpha)=\mathsf{F}_\mu(\alpha)={\mathsf B}_{\mu}^*(\alpha)=
B_{\mu}^*(\alpha).
$$
\end{enumerate}
\end{theorem}
\begin{proof}
All of the ideas needed to prove this theorem can be found in
\cite[Proposition 9]{BJ}, Propositions \ref{new_box}, \ref{mod_box}
and \ref{ourBB} and Theorem \ref{newform}.
\end{proof}

%%%%%%%%%%%%%%%%%%%%%%%%%%%%%%%%%%%%%%%%%%%%%%%%%%%%%%%%%%%%%%%%%%%%%%%%%%%%%%%%%%%%%
\subsection{Example 4}
%%%%%%%%%%%%%%%%%%%%%%%%%%%%%%%%%%%%%%%%%%%%%%%%%%%%%%%%%%%%%%%%%%%%%%%%%%%%%%%%%%%%%%%%

Given a class of exact dimensional measures (inhomogeneous
multinomial measures) whose support is the whole interval $[0, 1]$,
the multifractal functions $b_\mu$, $\mathsf b_\mu$, $\mathsf B_\mu$
and $B_\mu$ are real analytic and agree at two points only $0$ and
$1$ (for more details, see \cite{SH1}). These measures satisfy our
{\it multifractal formalism} in the sense that, for $\alpha$ in some
interval, the Hausdorff dimension of the level sets $E(\alpha)$ is
given by the Legendre transform of lower multifractal
Hewitt-Stromberg function and their packing dimension by the
Legendre transform of the upper multifractal Hewitt-Stromberg
function. More specifically,
$$
b_{\mu}(q)={\mathsf b}_{\mu}(q)< B_{\mu}(q)={\mathsf B}_{\mu}(q)
\quad\text{for all} \;\; q\notin \{0, 1\},
$$

$$
{f}_\mu(\alpha)=\mathsf{f}_\mu(\alpha)= b_{\mu}^*(\alpha)={\mathsf
b}_{\mu}^*(\alpha)
$$
and
$$
\qquad\qquad\qquad\mathsf{F}_\mu(\alpha)={F}_\mu(\alpha)
=\mathsf{B}_\mu^*(\alpha)={B}_\mu^*(\alpha)\;\;\text{for
some}\;\;\alpha.
$$

%%%%%%%%%%%%%%%%%%%%%%%%%%%%%%%%%%%%%%%%%%%%%%%%%%%%%%%%%%%%%%%%%%%%%%%%%%%%%%%%%%%%%%%%%%%%%%%%%%%%%%%%%%%%%%%%%%%%%%%%%%%%%%%%%%%%%%%%%
                                                                 \section{Open
                                                                 problems}\label{sec5}
                             %%%%%%%%%%%%%%%%%%%%%%%%%%%%%%%%%%%%%%%%%%%%%%%%%%%%%%%%%%%%%%%%%%%%%%%%%%%%%%%%%%%%%%

Motivated by some results and examples developed in \cite{Ol1, O3,
O2, O4, Ol, Zi}, we therefore ask the following questions.
\begin{enumerate}
\item Let $\mu\in\mathcal{P}_D(\mathbb{R}^n)$,
$E\subseteq\supp\mu$, $p,q\in\mathbb{R}$ and $\alpha\in[0,1].$ Then,
the following problem remains open:
 $$
{\mathsf b}_{\mu}^{\alpha p+(1-\alpha)q}(E)\leq\alpha {\mathsf
B}_{\mu}^p(E)+(1-\alpha){\mathsf b}_{\mu}^q(E)?
 $$
\item Let $q\in \mathbb{R}$ and assume that ${\mathsf  B}_{\mu}(q)={\mathsf
b}_{\mu}(q)$. Are the measures ${\mathsf{
H}^{q,t}_{\mu}}_{\llcorner_{\supp\mu}}$ and ${\mathsf{
P}^{q,t}_{\mu}}_{\llcorner_{\supp\mu}}$ proportional, i.e. does
there exists a constant $c_q > 0$ such that
$$
{\mathsf{ P}^{q,t}_{\mu}}_{\llcorner_{\supp\mu}}= c_q {\mathsf{
H}^{q,t}_{\mu}}_{\llcorner_{\supp\mu}}?
$$
Even though it seems rather unlikely that the lower and upper
multifractal Hewitt-Stromberg measures are proportional in general,
the ratio of the measures ${\mathsf{
H}^{q,t}_{\mu}}_{\llcorner_{\supp\mu}}$ and ${\mathsf{
P}^{q,t}_{\mu}}_{\llcorner_{\supp\mu}}$ might still be bounded. We
therefore ask the following question: Does there exists a number $0
< c_q < +\infty$ such that
$$
{\mathsf{ H}^{q,t}_{\mu}}_{\llcorner_{\supp\mu}}\leq{\mathsf{
P}^{q,t}_{\mu}}_{\llcorner_{\supp\mu}}\leq c_q {\mathsf{
H}^{q,t}_{\mu}}_{\llcorner_{\supp\mu}}?
$$
\item Let $q\in \mathbb{R}$ and assume that ${b}_{\mu}(q)={\mathsf
b}_{\mu}(q)$. Are the measures ${\mathsf{
H}^{q,t}_{\mu}}_{\llcorner_{\supp\mu}}$ and
${{\mathcal{H}}^{q,t}_{\mu}}_{\llcorner_{\supp\mu}}$ proportional,
i.e. does there exists a constant $C_q > 0$ such that
$$
{\mathsf{ H}^{q,t}_{\mu}}_{\llcorner_{\supp\mu}}= C_q
{{\mathcal{H}}^{q,t}_{\mu}}_{\llcorner_{\supp\mu}}?
$$
Even though it seems rather unlikely that the multifractal Hausdorff
measure and the lower multifractal Hewitt-Stromberg measure are
proportional in general, the ratio of the measures
${{\mathcal{H}}^{q,t}_{\mu}}_{\llcorner_{\supp\mu}}$  and ${\mathsf{
H}^{q,t}_{\mu}}_{\llcorner_{\supp\mu}}$ might still be bounded. We
therefore ask the following question: Does there exists a number $0
< C_q < +\infty$ such that
$$
{{\mathcal{H}}^{q,t}_{\mu}}_{\llcorner_{\supp\mu}}\leq{\mathsf{
H}^{q,t}_{\mu}}_{\llcorner_{\supp\mu}}\leq C_q
{{\mathcal{H}}^{q,t}_{\mu}}_{\llcorner_{\supp\mu}}?
$$
\item Let $p, q \in \mathbb{R}$ and assume that ${\mathsf
b}_{\mu}(q)$ is differentiable at $p$ and $q$ with $ {\mathsf
b}_{\mu}'(p) \neq{\mathsf b}_{\mu}'(q)$. Then, the following problem
remains open:
$$
{\mathsf{ H}^{p,{\mathsf b}_{\mu}(p)}_{\mu}}_{\llcorner_{\supp\mu}}
\bot \;\;{\mathsf{ H}^{q,{\mathsf
b}_{\mu}(q)}_{\mu}}_{\llcorner_{\supp\mu}}?
$$

\item Let $p, q \in \mathbb{R}$ and assume that ${\mathsf
B}_{\mu}(q)$ is differentiable at $p$ and $q$ with $ {\mathsf
B}_{\mu}'(p) \neq{\mathsf B}_{\mu}'(q)$. Then, the following problem
remains open:
$$
{\mathsf{ P}^{p,{\mathsf B}_{\mu}(p)}_{\mu}}_{\llcorner_{\supp\mu}}
\bot \;\;{\mathsf{ P}^{q,{\mathsf
B}_{\mu}(q)}_{\mu}}_{\llcorner_{\supp\mu}}?
$$
\item Is it true that the weaker condition ${b}_{\mu}(q)={\mathsf b}_{\mu}(q)$ is sufficient to obtain the
conclusion of Theorem \ref{newform}?\\
\item  Let $\mu\in \mathcal{P}_D(\mathbb{R}^n)$,
$\nu\in \mathcal{P}_D(\mathbb{R}^m)$ and $q,s,t\in \mathbb{R}$.
Assume that
 $c>0$, $E\subseteq \mathbb{R}^n$, $F\subseteq \mathbb{R}^m$,
$H\subseteq \mathbb{R}^{n+m}$ and $H(y)=\big\{x;\; (x,y)\in
H\big\}$. Then, the following problem remains open:
\begin{eqnarray*}
\int {{\mathsf{ H}}}^{q,s}_{\mu}(H(y))~~d{{\mathsf{
H}}}^{q,t}_{\nu}(y)&\leq& c~~ {{\mathsf{ H}}}^{q,s+t}_{\mu\times\nu
}(H)
\end{eqnarray*}
\begin{eqnarray*}
{{\mathsf{ H}}}^{q,s+t}_{\mu\times\nu }(E\times F)&\leq&
c~~{{\mathsf{ H}}}^{q,s}_{\mu}(E)~~{{\mathsf{ P}}}^{q,t}_{\nu}(F)
\end{eqnarray*}
\begin{eqnarray*}
\int {{\mathsf{ H}}}^{q,s}_{\mu}(H(y))~~d{{\mathsf{
P}}}^{q,t}_{\nu}(y)&\leq& c~~ {{\mathsf{ P}}}^{q,s+t}_{\mu\times\nu
}(H)
\end{eqnarray*}
\begin{eqnarray*}
{{\mathsf{ P}}}^{q,s+t}_{\mu\times\nu }(E\times F)&\leq&
c~~{{\mathsf{ P}}}^{q,s}_{\mu}(E)~~{{\mathsf{ P}}}^{q,t}_{\nu}(F)
\end{eqnarray*}
and
$$
{\mathsf b}_{\mu}^q(E)+{\mathsf b}_{\nu}^q(F)\leq {\mathsf
b}_{\mu\times \nu}^q(E\times F) \leq{\mathsf b}_{\mu}^q(E)+{\mathsf
B}_{\nu}^q(F) \leq {\mathsf B}_{\mu\times \nu}^q(E\times F) \leq
{\mathsf B}_{\mu}^q(E)+{\mathsf B}_{\nu}^q(F)?
$$

\item The multifractal Hausdorff dimension function $b_\mu$ and the
lower multifractal Hewitt-Stromberg function ${\mathsf b}_\mu$ do
not necessarily coincide. Motivated by the results developed in
\cite{MiOl}, we conjecture that there exist Borel probability
measures $\mu$ on $\mathbb{R}^n$ such that
$$
b_\mu\neq {\mathsf b}_{\mu}={\mathsf B}_{\mu}={
B}_{\mu}\quad\text{or}\quad b_\mu< {\mathsf b}_{\mu}<{\mathsf
B}_{\mu}={ B}_{\mu}< { \Lambda}_{\mu}\quad\text{for all}\quad q\neq
1.
$$
In particular, this will imply that
$$
{f}_\mu(\alpha)=
b_{\mu}^*(\alpha)<\mathsf{f}_\mu(\alpha)=\mathsf{F}_\mu(\alpha)={F}_\mu(\alpha)=\mathsf{b}_\mu^*(\alpha)
=\mathsf{B}_\mu^*(\alpha)={B}_\mu^*(\alpha)\;\;\text{for
some}\;\;\alpha\geq0.
$$
\end{enumerate}

%%%%%%%%%%%%%%%%%%%%%%%%%%%%%%%%%%%%%%%%%%%%%%%%%%%%%%%%%%%%%%%%%%%%%%%%%%%%%%%%%%%%%%%%%%%%%%%%%%%%%%%%%%%%%%%%%%%%%%%%%%%%%%%%%%%%
                                                                \section*{Acknowledgments}
%%%%%%%%%%%%%%%%%%%%%%%%%%%%%%%%%%%%%%%%%%%%%%%%%%%%%%%%%%%%%%%%%%%%%%%%%%%%%%%%%%%%%%%%%%%%%%%%%%%%%%%%%%%%%%%%%%%%%%%%%%%%%%%%%%%%%ù

The authors would like to thank Professor {\it Lars Olsen} for his
first reading of this work, the interest he gave to it and his
valuable comment which improves the presentation of the paper. And
they thank the anonymous referees for their valuable comments and
suggestions that led to the improvement of the manuscript.
%%%%%%%%%%%%%%%%%%%%%%%%%%%%%%%%%%%%%%%%%%%%%%%%%%%%%%%%%%%%%%%%%%%%%%%%%%%%%%%%%%%%%%%%%%%%%%%%%%%%%%%%%%%%%%%%%%%%%%%%%%%%%%%%%%%%%%%%%%%%%%%%%%%%%%%%%%%%

%%%%%%%%%%%%%%%%%%%%%%%%%%%%%%%%%%%%%%%%%%%%%%%%%%%%%%%%%%%%%%%%%%%%%%%%%%%%%%%%%%%%%%%%%%%%%%%%%%%%%%%%%%%%%%%%%%%%%%%%%%%%%%%%%%%%%%%%%%%%%%%%%%%%%%%%%%%%
%%%%%%%%%%%%%%%%%%%%%%%%%%%%%%%%%%%%%%%%%%%%%%%%%%%%%%%%%%%%%%%%%%%%%

%%%%%%%%%%%%%%%%%%%%%%%%%%%%%%%%%%%%%%%%%%%
%%%%%%%%%%%%%%%%%%%%%%%%%%%%%%%%%%%%%%%%%

\end{document}